\documentclass[10 pt]{amsart}
\usepackage{amssymb,amsmath,amsthm,amscd,mathrsfs, color}
\usepackage{pdfsync}
\usepackage{graphicx}
\usepackage[matrix,arrow,tips,curve]{xy}
\usepackage[english]{babel}
\usepackage[leqno]{amsmath}
\usepackage{amssymb,amsthm}
\usepackage{amscd}
\usepackage{enumerate}
\usepackage{fullpage}

\numberwithin{equation}{section}
\newtheorem{teo}{Theorem}[section]
\newtheorem{pro}[teo]{Proposition}
\newtheorem{lem}[teo]{Lemma}

\newtheorem{cor}[teo]{Corollary}

\newtheorem{teoalpha}{Theorem}

\newtheorem{coralpha}[teoalpha]{Corollary}

\theoremstyle{definition}
\newtheorem{defi}[teo]{Definition}
\newtheorem{exa}[teo]{Example}

\theoremstyle{remark}
\newtheorem{rem}[teo]{Remark}

\newenvironment{sis}{\left\{\begin{aligned}}{\end{aligned}\right.}

\newcommand{\w}{\widetilde}
\newcommand{\ov}{\overline}
\newcommand{\un}{\underline}

\newcommand{\wh}{\widehat}
\newcommand{\wt}{\widetilde}

\newcommand{\bbA}{{\mathbb A}}

\newcommand{\bbC}{{\mathbb C}}
\newcommand{\bbD}{{\mathbb D}}

\newcommand{\bbH}{{\mathbb H}}

\newcommand{\bbM}{{\mathbb M}}
\newcommand{\bbN}{{\mathbb N}}

\newcommand{\bbP}{{\mathbb P}}
\newcommand{\bbQ}{{\mathbb Q}}
\newcommand{\bbR}{{\mathbb R}}

\newcommand{\bbT}{{\mathbb T}}

\newcommand{\bbZ}{{\mathbb Z}}
\newcommand{\bbz}{{\mathbb Z}}

\newcommand{\Aut}{\operatorname{Aut}}
\newcommand{\Bir}{\operatorname{Bir}}
\newcommand{\Supp}{\operatorname{Supp}}

\newcommand{\Spec}{\operatorname{Spec}}
\newcommand{\Spf}{\operatorname{Spf}}
\newcommand{\Def}{\operatorname{Def}}
\newcommand{\Sym}{\operatorname{Sym}}
\newcommand{\Id}{\operatorname{Id}}

\newcommand{\Pic}{\operatorname{Pic}}

\newcommand{\conv}{\operatorname{conv}}
\newcommand{\Stab}{{\rm Stab}}

\newcommand{\id}{\operatorname{id}}
\newcommand{\rank}{\operatorname{rank}}

\newcommand{\II}{\mathbb{I}}
\newcommand{\age}{\operatorname{age}}

\newcommand{\Gm}{\mathbb{G}_m}
\newcommand{\GL}{\operatorname{GL}}
\newcommand{\SL}{\operatorname{SL}}

\newcommand{\OO}{\mathcal{O}}

\newcommand{\m}{\mathfrak{m}}

\newcommand{\Mgbar}{\ov{\operatorname{M}}_g}

\renewcommand{\OE}{\overset{\rightarrow}{E}}
\newcommand{\OS}{\overset{\rightarrow}{S}}
\newcommand{\el}[1][e]{\overset{\leftarrow}{#1}}
\newcommand{\er}[1][e]{\overset{\rightarrow}{#1}}

\newcommand{\Cir}{\operatorname{Cir}}
\newcommand{\Loops}{\operatorname{Loops}}

\begin{document}

\title[The universal compactified Jacobian]{The singularities and birational geometry of the  universal compactified Jacobian}

\author[Casalaina-Martin]{Sebastian Casalaina-Martin}
\address{University of Colorado, Department of Mathematics,
Campus Box 395, Boulder, CO 80309, USA}
\email{casa@math.colorado.edu}

\author[Kass]{Jesse Leo Kass}
\address{Department of Mathematics, University of South Carolina, 1523 Greene Street, Columbia, SC 29208, USA}
\email{kassj@math.sc.edu}

\author[Viviani]{Filippo Viviani}
\address{ Dipartimento di Matematica,
Universit\`a Roma Tre,
Largo S. Leonardo Murialdo 1,
00146 Roma, Italy}
\email{viviani@mat.uniroma3.it}

\thanks{The first author was supported by NSF grant DMS-1101333.  The second author was supported by NSF grant DMS-0502170. The third author  
has been partially  supported by  the MIUR--FIRB project \textit{Spazi di moduli e applicazioni}, by CMUC and by the FCT-grants   PTDC/MAT-GEO/0675/2012 and EXPL/MAT-GEO/1168/2013.}

\date{\today}



\begin{abstract}
In this paper we establish that the singularities of the universal compactified Jacobian are  canonical if the genus is at least four.
As a corollary we determine the Kodaira dimension and the Iitaka fibration  of the universal compactified Jacobian for every  degree and genus.  We also determine   the birational automorphism group for every degree if the genus is  at least twelve.  This extends work of G. Farkas and A. Verra, as well as  that of G. Bini, C. Fontanari and the third author.
\end{abstract}

\maketitle

\bibliographystyle{amsalpha}

\section*{Introduction}

Jacobians of  non-singular curves are  principally polarized abelian varieties, which from the perspective of birational geometry   are among the simplest varieties.    On the other hand, for a family of non-singular curves, the relative Jacobian may exhibit more interesting birational behavior,  not necessarily  reflective of the birational geometry of the base.  For instance, over the moduli space of non-singular, genus $g\ge 2$, automorphism free curves $M_g^\circ$ there is a universal curve $C_g^\circ$, and consequently a universal Jacobian $\operatorname{Pic}^0(C_g^\circ/M_g^\circ)$.     In this paper, we investigate the birational geometry of this space and show for instance, that the Kodaira dimension of $\operatorname{Pic}^0(C_g^\circ/M_g^\circ)$ can be different from the Kodaira dimension of $M_g^\circ$.

More generally, for any integer $d$,  Caporaso \cite{caporaso} (see also \cite{Pan}) has constructed  a universal compactified Jacobian $\pi:\bar J_{d,g}\to \overline M_g$ over the moduli space of Deligne--Mumford stable curves;  this space has fiber over a non-singular, automorphism free curve $C$ given by the degree $d$ Jacobian $J^dC$.  In particular $\bar J_{0,g}$ provides a compactification of the universal Jacobian.
In this paper we focus on two main problems concerning the birational geometry of these spaces, namely determining the Kodaira dimension, and  determining the birational automorphism group.
These problems go back at least to Caporaso's work, and have been investigated recently by Farkas and Verra \cite{farkasverra} and  Bini, Fontanari and the third author \cite{BFV} in special cases.

Due to the work of \cite{BFV}, the main point needed to answer these questions in full generality is to provide a good description of the local structure of $\bar J_{d,g}$.  In this paper, we investigate this question in detail, providing an explicit description of the complete local ring at a point, as well as formulas for various invariants of the ring in terms of the dual graphs of the curves.  In particular, we
establish that $\bar J_{d,g}$ has canonical singularities.

\begin{teoalpha}\label{teoCS}
Assume that ${\rm char}(k)=0$. If $g\geq 4$, then the universal compactified  Jacobian $\bar J_{d,g}$ has  canonical  singularities for any $d\in \bbZ$.
\end{teoalpha}

The arguments   build on the previous work of the authors in two ways.  First, extending the deformation theory in \cite{CMKVb}, we are able to reduce the problem to the study of a special class of combinatorial rings, called cographic toric face rings,  investigated  in \cite{CMKVa}.  In full generality, these rings can exhibit  poor behavior (see \cite[\S 5.1]{CMKVa}).  However, as it turns out, the rings appearing from the deformation theory of the \emph{universal} compactified Jacobian form a special class of rings with mild singularities.
 The specific cographic rings appearing in this paper will be denoted by $U(\Gamma)$  and are defined from the data of a  graph $\Gamma$ (Definition \ref{D:cogringvar}).  Our  main result for these rings is the following theorem.

\begin{teoalpha} \label{teoUG}
Let $\Gamma$ be a finite, connected graph and let $k$ be an algebraically closed field.  The cographic toric ring $U(\Gamma)$ is a finitely generated, integral $k$-algebra and the singularities of the associated variety $\operatorname{Spec}U(\Gamma)$ are Gorenstein, rational, and terminal.
\end{teoalpha}

Using the results in \cite{CMKVa}, together with standard results on toric varieties, we are also able to establish a number of further properties of the rings $U(\Gamma)$ (and consequently $\bar J_{d,g}$) in terms of invariants of the graph $\Gamma$, including the dimension (Corollary \ref{CorDim}), the dimension of the tangent space (Proposition \ref{protan}), and the multiplicity (Theorem \ref{T:multiGKZ}).

From Theorem \ref{teoCS} and the work of Bini--Fontanari--Viviani, one obtains the following consequence for the birational geometry of $\bar J_{d,g}$.

\begin{coralpha}\label{teoKD}
Assume that ${\rm char}(k)=0$.
The Kodaira dimension of the universal Jacobian $\bar J_{d,g}$ is given by
$$\kappa(\bar J_{d,g})=
\begin{cases}
-\infty & \text{ if } g\leq 9, \\
0 & \text{ if } g=10, \\
19 & \text{ if } g= 11, \\
3g-3 & \text{ if } g\geq 12.
\end{cases}
$$
Moreover, for $g\ge 10$, the Iitaka fibration of $ J_{d,g}$ is given as follows:
\begin{enumerate}
\item For $g\ge 12$, the Iitaka fibration is the forgetful  morphism $\pi:\bar J_{d,g}\to \overline M_g$.

\item For $g=11$, the Iitaka fibration is the rational map $\bar J_{d,11}\dashrightarrow \mathcal F_{11}$, where $\mathcal F_{g}$ is the moduli of K3 surfaces with polarization of degree $2g-2$, and the rational map takes a  general pair $(C,L)$ to the pair $(S,\mathcal O_S(C))$, where $S$ is the unique K3 containing $C$ (see \cite{mukai96}).
  
\item For $g=10$, the Iitaka fibration is the structure  morphism $\bar J_{d,10}\to \operatorname{Spec} k$.  
\end{enumerate}

\end{coralpha}

For $g=22$ and $g\ge 24$ the statement on the Kodaira dimension follows from general results in birational geometry, together with well-known results for $\overline M_g$  (see Remark \ref{remUK}).
In the remaining range, the result  was proven by Bini--Fontanari--Viviani \cite[Thm.~1.2]{BFV} under the  numerical condition that
$\gcd(d+1-g,2g-2)=1$ or $g=23$, and by Farkas--Verra \cite{farkasverra}  in the special case $d=g$.
In particular, the case $d=0$ was not known.  We also point out that while we have obtained here a complete classification of the Kodaira dimension  for the universal Jacobian,  the Kodaira dimension of  the moduli of curves is still unknown in the range $17\le  g\le 21$, $g=23$.    Finally, for $10\le g\le  16$, we have $\kappa(\bar J_{d,g})\ne \kappa(\overline{M}_g)$.  We direct the reader to \eqref{eqnKDC} for more details, as well as Remark \ref{remFVthch}, which compares these numerics with the recent work of Farkas--Verra  \cite{Far-theta, FVNik, farkasverraOddThCh, farkasThCh}  on the moduli space of theta characteristics.

  Another immediate observation is that the Kodaira dimension is independent of $d$.   One might guess the reason for this is that $\bar J_{d,g}$ is birational  $\bar J_{d',g}$ for different $d$ and $d'$.   Our next result shows this is not generally the case.

\begin{coralpha}\label{teoBA}
Assume that ${\rm char}(k)=0$ and that $g\geq 12$. If $\eta: J_{d,g}\dashrightarrow J_{d',g}$ is a birational map, then
$d'=\pm d+n(2g-2)$ and $\eta$ is given by the map sending $(C,L)\in J_{d,g}$ into
$(C,L^{\pm 1}\otimes \omega_C^n)\in J_{d',g}$. In particular:
\begin{enumerate}[(i)]
\item $J_{d,g}$ is birational to $J_{d',g}$ if and only if $d'\equiv \pm d \mod 2g-2$.
\item The group $\Bir(J_{d,g})$ of birational automorphisms of $J_{d,g}$ is given by
$$\Bir(J_{d,g})=
\begin{cases}
\bbZ/ 2\bbZ & \text{ if } d=n(g-1)  \text{ for some } n\in \bbZ,\\
\{\Id \} & \text{ otherwise.}
\end{cases}
$$
Moreover, if $d=n(g-1)$ for some $n\in \bbZ$ then the generator of $\Bir(J_{d,g})$ is the birational automorphism sending $(C,L)$ into $(C, L^{-1}\otimes \omega_C^n)$.
\end{enumerate}
\end{coralpha}

This was proven by Bini--Fontanari--Viviani \cite[Thm.~1.7]{BFV} in the special case
$\gcd(d+1-g,2g-2)=1$ (or $g\ge 22$), and builds on work of Caporaso \cite{caporaso}.

The paper is organized as follows.  In Section 1 we review terminology concerning graphs, and various constructions with graphs that will appear later.  In Section 2 we define the combinatorial rings $U(\Gamma)$ and establish some first properties of the rings.  In Section 3 we establish some specific presentations of the ring, which are useful for later computations, and also for connecting the rings with deformations.   In Section 4 we discuss the singularities of the rings $U(\Gamma)$.   In Section 5 we describe the rings as invariants for a group action, which provides the framework for the connection with deformations of sheaves.   In Section 6 we provide some examples of these rings.    In Section 7 we make the connection with the universal compactified Jacobian, and establish the results on the singularities, Kodaira dimension, and birational automorphism group. 

The paper ends with an appendix in which we investigate the singularities of finite quotients of toric varieties. More specifically, the  focus is on establishing a  
Reid--Tai--Shepherd-Barron criterion for singular toric varieties; i.e., a numerical condition that can be used to determine when a finite quotient of a singular toric variety has canonical, or terminal singularities.       The main result is Proposition \ref{P:crit-cyc}, which in conjunction with Theorem \ref{T:redcyc}, is a direct generalization of the Reid--Tai--Shepherd-Barron criterion.   While we expect the  generalization is well-known to the experts, we were not aware of a reference, and include proofs here.

\subsection*{Acknowledgements}
The first author would like to thank Jonathan Wise for conversations on toric geometry and deformation theory, and James McKernan for a discussion on singularities of toric varieties.

\section{Preliminaries on graphs} \label{S:Prelim}

In this section we introduce some constructions on graphs that we will  use in this paper.

\subsection{Graph notation}

Following Serre  \cite[\S~2.1]{SerT}, a \emph{graph} $\Gamma$ consists of the data
$(\OE  \xymatrix{ \ar @{->}^s @< 2pt> [r] \ar@{->}_t @<-2pt> [r] &  } V, \OE \stackrel{\iota}{\to} \OE),$
where $V$ and $\OE$ are sets, $\iota$ is a fixed-point free involution, and $s$ and $t$ are maps satisfying $s(\er)=t(\iota(\er))$ for all $\er \in \OE$.  The maps $s$ and $t$ are called the \emph{source} and \emph{target} maps respectively.  We call $V=:V(\Gamma)$ the set of \emph{vertices}.  We call $\OE=:\OE(\Gamma)$ the set of \emph{oriented edges}.

 We define the set of \emph{(unoriented) edges} to be  $E(\Gamma)=E:=\OE/\iota$.   Given an oriented edge $\er \in \OE $ we will denote by $\underline \er$ the class of $\er$ in  $E$.   An \emph{orientation of an edge} $e\in E$ is a representative for $e$ in $\OE$; we use the notation $\er$ and $\el$ for the two possible orientations of $e$.
  An \emph{orientation of a graph $\Gamma$} is a section $\phi:E\to \OE$ of the quotient map.  An \emph{oriented graph}  consists of a pair $(\Gamma,\phi)$ where $\Gamma$ is a graph and $\phi$ is an orientation.    Given an oriented graph, we say that $\phi(e)$ is the \emph{positive orientation} of the edge $e\in E$. Given a subset $S\subseteq E$, we define $\OS\subseteq \OE$ to be the set of all orientations of the edges in $S$.

We will say that  two edges of a graph are \emph{parallel} if they connect the same (not necessarily distinct) vertices.
We say that an edge of a connected graph is a \emph{separating} edge if removing the edge disconnects the graph.
Two edges of a connected graph are a \emph{separating pair}  if they are both non-separating edges and if removing the two edges disconnects the graph.

If $\Gamma$ is connected, then we say that an orientation $\phi$ of $\Gamma$ is \emph{totally
	cyclic} if there does not exist a proper non-empty subset $W\subset V(\Gamma)$ such that the edges between $W$ and
	its complement  $V(\Gamma)\smallsetminus W$ all go in the same direction (i.e.~either all these edges are oriented from $W$ to $V(\Gamma)\smallsetminus W$ or all are oriented in the opposite direction).
	If $\Gamma$ is disconnected, then we say that an orientation of $\Gamma$ is totally cyclic if the orientation induced on each connected component of $\Gamma$ is totally cyclic.
	
	A graph $\Gamma$ is called  \emph{cyclic} if it is connected, free from separating edges,
and satisfies $b_1(\Gamma):=|E(\Gamma)|-|V(\Gamma)|+1=1$.   We will also call a cyclic graph a \emph{circuit}.  A cyclic graph together with a totally cyclic orientation is called an \emph{oriented circuit}.
A \emph{loop} is a circuit with a single edge.

\subsection{Ordinary homology and oriented homology}\label{S:HomGra}

Given any graph $\Gamma$, we can form its ordinary homology (which coincides with the homology of the underlying topological space) and its oriented homology.

Let $\mathbb C_0(\Gamma,\mathbb Z)$ be the free $\mathbb Z$-module with basis $V(\Gamma)$, let $\mathbb C_1(\Gamma,\mathbb Z)$ be the free $\mathbb Z$-module generated by $\OE(\Gamma)$ and consider the  boundary map $ \mathbb D $  defined as:
\begin{equation}\label{eqnCCcompS}
	\mathbb D: {\mathbb C}_1(\Gamma,\mathbb Z) \to {\mathbb C}_0(\Gamma,\mathbb Z)
	\end{equation}
	\begin{equation*}
	\er  \mapsto t(\er)-s(\er).
\end{equation*}
 We will denote by $\mathbb H_\bullet (\Gamma,\mathbb Z)$ the groups obtained from the homology of
 ${\mathbb C}_\bullet (\Gamma,\mathbb Z)$ and we will call them the \emph{oriented homology groups of $\Gamma$}.
 Let $(\ , \ )$ be the unique scalar product on $\mathbb C_1(\Gamma,\mathbb R)=\mathbb C_1(\Gamma,\mathbb Z)\otimes_{\mathbb Z}\mathbb R$ (and also its restriction to $\bbH_1(\Gamma,\bbZ)$ )
 such that the elements of $\OE(\Gamma)$ form an orthonormal basis.

Let $C_0(\Gamma,\bbZ)=\bbC_0(\Gamma,\bbZ)$,  let $C_1(\Gamma,\mathbb Z)$ be the quotient of $\bbC_1(\Gamma,\bbZ)$ by the relation $\el=-\er$ for every $e\in E(\Gamma)$
and consider the boundary map
\begin{equation}\label{eqnCCcomp}
	 \partial: C_1(\Gamma,\mathbb Z) \to C_0(\Gamma,\mathbb Z)
	\end{equation}
	\begin{equation*}
	[\er]  \mapsto t(\er)-s(\er),
\end{equation*}
where we denote by $[\er]$ the class of $\er$ in $C_1(\Gamma,\bbZ)$.
 We will denote by  $H_\bullet (\Gamma,\mathbb Z)$ the groups obtained from the homology of  $C_\bullet (\Gamma,\mathbb Z)$ and we will call them the \emph{ordinary homology groups of $\Gamma$}.
Note that $H_\bullet (\Gamma,\mathbb Z)$ is isomorphic to the homology of the underlying topological space of $\Gamma$.
Let $(\ ,\ )$ be the unique scalar product on $C_1(\Gamma,\bbZ)$ (and also its restriction to $H_1(\Gamma,\bbZ)$) such that
$$
\begin{aligned}
& ([\er],[\er])=-([\er],[\el])=1  \hspace{0.3cm} \text{Êfor any }Êe\in E, \\
&  ([\er_1],[\er_2])=0 \hspace{0.3cm} \text{Êfor any } \er_1,\er_2\in \OE \text{ such that } [\er_1]\neq \pm [\er_2].
\end{aligned}
$$

For a connected graph $\Gamma$,  the corank of the image of $\bbD$ (resp.~of $\partial$) is one.
Consequently, for a connected graph, we have
\begin{equation}\label{E:dim-H}
\begin{aligned}
& \operatorname{rank} H_1(\Gamma,\bbZ)=|E(\Gamma)|-|V(\Gamma)|+1=:b_1(\Gamma), \\
& \operatorname{rank} \bbH_1(\Gamma,\bbZ)=2|E(\Gamma)|-|V(\Gamma)|+1=b_1(\Gamma)+|E(\Gamma)|. \\
\end{aligned}
\end{equation}

In order to determine the relationship between ordinary and oriented homology, consider the following commutative diagram
 \begin{equation}\label{E:diag-hom}
\xymatrix{
      \bbC_1(\Gamma,\bbZ)  \ar@{->>}[d]^{}\ar[rr]^{\mathbb D} && \mathbb C_0(\Gamma,\bbZ)\ar@{=}[d]\\
     C_1(\Gamma,\bbZ)  \ar[rr]^{\partial} && C_0(\Gamma,\bbZ)\\
  }
\end{equation}
where the left vertical map send $\er$ into $[\er]$. The above diagram \eqref{E:diag-hom} induces an equality $\mathbb H_0(\Gamma,\mathbb Z)=H_0(\Gamma,\mathbb Z)$ and 
a surjection $\mathbb H_1(\Gamma,\mathbb Z)\twoheadrightarrow H_1(\Gamma,\mathbb Z)$, whose kernel can be described as follows.

\begin{lem}\label{L:kerH1}
The kernel of the natural surjection $\mathbb H_1(\Gamma,\mathbb Z)\twoheadrightarrow H_1(\Gamma,\mathbb Z)$ is generated by $\{\er +\el\}_{e\in E(\Gamma)}$.
\end{lem}
\begin{proof}
From the definition of $\mathbb D$, we have  $\er+\el\in \mathbb H_1(\Gamma,\mathbb Z)$.  Clearly $\er+\el$ also maps to zero in $C_1(\Gamma,\mathbb Z)$.   On the other hand, suppose that $\sum_{e\in E} (a_e\er +b_e\el)\in \bbH_1(\Gamma,\bbZ)$ is in the kernel of the above map.  Then by definition $\sum_{e\in E} (a_e-b_e)[\er]=0$, and so $a_e=b_e$ for all $e\in E$, since $\{[\er]\}$ is a basis for $C_1(\Gamma,\mathbb Z)$.
\end{proof}

\subsection{Doubled graphs and doubled orientations}\label{S:double}

In this section, we introduce a class of graphs, called doubled graphs, together with canonical totally cyclically orientations of them, called doubled orientations,  which are obtained from a graph by doubling its edges.

\begin{defi}
Let $\Gamma$ be a connected graph.  Define the \emph{doubled graph} of $\Gamma$, denoted $\Gamma^d$, to  be the graph obtained by doubling the edges of $\Gamma$; i.e.~$\Gamma^d$ is the graph obtained from $\Gamma$ by replacing each edge $e$ of $\Gamma$
with a pair of parallel edges $e'$ and $e''$ of $\Gamma^d$ having the same endpoints as  $e$ (see Figure \ref{Fig:dg}).
 To be precise,
$
V(\Gamma^d)=V(\Gamma)$, $\OE(\Gamma^d)=\bigcup_{\er \in \OE} \{\er',\er''\}$, and we define  $s(\er')=s(\er '')=s(\er)$, $t(\er')=t(\er'')=t(\er)$ and $\iota (\er')=\el'$, $\iota(\er'')=\el''$.
Note that
$$
E(\Gamma^d)=\bigcup_{e\in E(\Gamma)}\{e',e''\}
$$
where we use the convention that if $e=\underline \er$, then $e'=\underline \er '$, $e''=\underline \er''$.
\end{defi}

\begin{figure}
\begin{equation*}
\xymatrix{
&&\text{Unoriented edges} \ E&&&\text{Oriented edges} \ \OE &&&\\
\Gamma & *{\bullet}   \ar @{-}[rr]|-{\SelectTips{cm}{}\object@{}}^{e}
  &&*{\bullet}& *{\bullet} \ar @{-}@/_1pc/[rr]|-{\SelectTips{cm}{}\object@{<}}_{\el}
 \ar@{-} @/^1pc/[rr]|-{\SelectTips{cm}{}\object@{>}}^{\er}    &&*{\bullet}&&\\
 &\\
\Gamma^d&   *{\bullet} \ar @{-}@/_1pc/[rr]|-{\SelectTips{cm}{}\object@{}}_{e''}
 \ar@{-} @/^1pc/[rr]|-{\SelectTips{cm}{}\object@{}}^{e'}    &&*{\bullet}&*{\bullet} \ar @{-}@/_1pc/[rr]|-{\SelectTips{cm}{}\object@{<}}^{\el'} \ar @{-} @/_2.5pc/[rr] |-{\SelectTips{cm}{}\object@{<}}^{\el''}
 \ar@{-} @/^1pc/[rr]|-{\SelectTips{cm}{}\object@{>}}^{\er'}  \ar@{-}@/^2.5pc/[rr]|-{\SelectTips{cm}{}\object@{>}}^{\er''}^<{v_1}^>{v_2}  &&*{\bullet}
}
\end{equation*}
\hspace{1cm}
\caption{Doubled graph.
}\label{Fig:dg}
\end{figure}
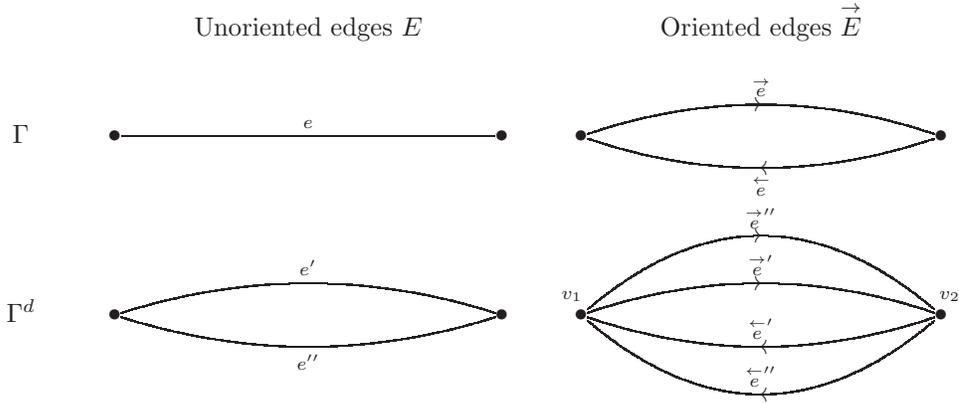

The graph  $\Gamma^d$ drawn with its unoriented edges looks like the  graph  $\Gamma$ drawn with its oriented edges (see Figure \ref{Fig:dg}).   In this way, choosing an identification of edges gives an orientation $\phi^d$ of $\Gamma^d$.   In fact, given an orientation $\phi$ of $\Gamma$, one obtains an orientation $\phi^d$ of $\Gamma^d$ by orienting each edge $e'$ in the same direction as $\phi(e)$, and each edge $e''$ in the opposite direction (see Figure \ref{Fig:do}).  
More precisely:

 \begin{defi}
Given an orientation $\phi$ of $\Gamma$, define the \emph{doubled  orientation} $$\phi^d:E(\Gamma^d)\to \OE(\Gamma^d)$$
$$
\phi^d(e')=\phi(e)'
$$
$$
\phi^d(e'')=\iota (\phi(e)'')
$$
\end{defi}

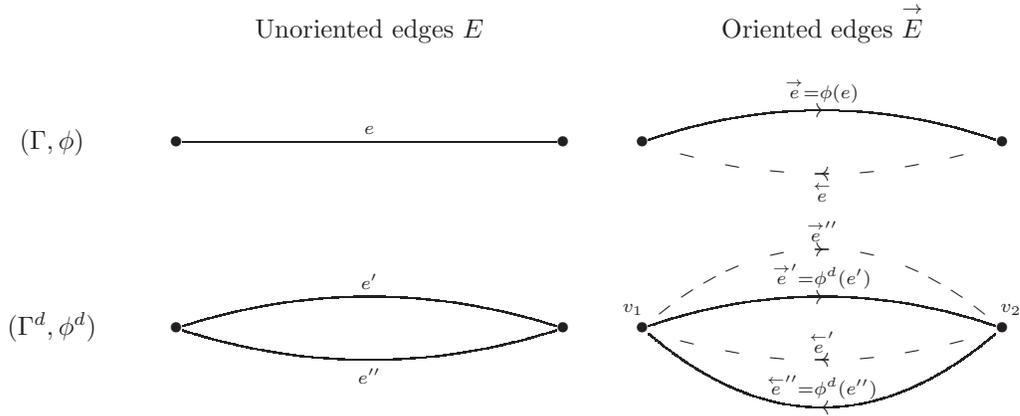
\begin{figure} 
\begin{equation*}
\xymatrix{
&&\text{Unoriented edges}\ E&&&\text{Oriented edges} \ \OE &&&\\
(\Gamma,\phi) & *{\bullet}   \ar @{-}[rr]|-{\SelectTips{cm}{}\object@{}}^{e}
  &&*{\bullet}& *{\bullet} \ar @{--}@/_1pc/[rr]|-{\SelectTips{cm}{}\object@{<}}_{\el}
 \ar@{-} @/^1pc/[rr]|-{\SelectTips{cm}{}\object@{>}}^{\er=\phi(e)}    &&*{\bullet}&&\\
 &\\
(\Gamma^d,\phi^d)&   *{\bullet} \ar @{-}@/_1pc/[rr]|-{\SelectTips{cm}{}\object@{}}_{e''}
 \ar@{-} @/^1pc/[rr]|-{\SelectTips{cm}{}\object@{}}^{e'}    &&*{\bullet}&*{\bullet} \ar @{--}@/_1pc/[rr]|-{\SelectTips{cm}{}\object@{<}}^{\el'} \ar @{-} @/_2.5pc/[rr] |-{\SelectTips{cm}{}\object@{<}}^{\el''=\phi^d(e'')}
 \ar@{-} @/^1pc/[rr]|-{\SelectTips{cm}{}\object@{>}}^{\er'=\phi^d(e')}  \ar@{--}@/^2.5pc/[rr]|-{\SelectTips{cm}{}\object@{>}}^{\er''}^<{v_1}^>{v_2}  &&*{\bullet}
}
\end{equation*}
\hspace{1cm}
\caption{Doubled orientation}\label{Fig:do}
\end{figure}

\begin{lem}\label{L:doub-or}
The doubled orientation $\phi^d$ on $\Gamma^d$ is canonical, i.e.~it does not depend on the choice of $\phi$ up to automorphisms of $\Gamma^d$, and it is totally cyclic.
\end{lem}

\begin{proof}
Choose an (unoriented) edge $f\in E(\Gamma)$, define a new orientation $\phi^f$ of $\Gamma$ by reversing the orientation on $f$; i.e.~setting
$$\phi^f(e)=\begin{sis}
\iota(\phi(f)) & \text{ if }Êe=f, \\
\phi(e) & \text{ if } e\neq f.
\end{sis}
$$
Define an automorphism $\psi$ of $\Gamma^d$ that  is the identity on vertices, exchanges $f'$ and $f''$ and fixes $e'$ and $e''$ for all other edges $e\neq f$ of $\Gamma$.
Then clearly $\psi$ will send the orientation $\phi^d$ into $(\phi^f)^d$. Since every other orientation of $\Gamma$ can be obtained from $\phi$ by iteratively applying the above construction, we have
shown that $\phi^d$ is canonical.

The fact that $\phi^d$ is totally cyclic follows easily from the fact that each pair of parallel (unoriented) edges $e'$ and $e''$ of $\Gamma^d$ associated to an edge $e$ of $\Gamma$ are given opposite orientations by $\phi^d$.
\end{proof}

The oriented homology of $\Gamma$ is canonically isomorphic to the ordinary homology of $\Gamma^d$. In order to prove this, fix an orientation $\phi$ of $\Gamma$ and consider the diagram
 \begin{equation}\label{E:ori-double}
\xymatrix{
    [\phi^d(e')]\in \ar@{|->}[d] &  C_1(\Gamma^d,\bbZ)  \ar[d]^{\cong}\ar[rr]^{\partial} && \mathbb C_0(\Gamma^d,\bbZ)\ar@{=}[d]\\
    \phi(e) \in & \bbC_1(\Gamma,\bbZ)  \ar[rr]^{\bbD} && \bbC_0(\Gamma,\bbZ)\\
  }
\end{equation}
where the left vertical map is the group isomorphism obtained by, for each $e\in E(\Gamma)$, sending  $[\phi^d(e')]\in C_1(\Gamma^d,\bbZ) $ into $\phi(e)\in \mathbb C_1(\Gamma,\mathbb Z)$ (and $[\phi^d(e'')]$ to $ \iota \phi(e)$).  In short, choosing a doubled orientation  $\phi^d$ on $\Gamma^d$, then  $C_1(\Gamma^d,\mathbb Z)$ can be given a basis consisting of the oriented edges determined by $\phi^d$; these edges  are in bijection (including orientation) with the collection of \emph{all} oriented edges of $\Gamma$, which form a basis of $\mathbb C_1(\Gamma,\mathbb Z)$ (see  Figure \ref{Fig:do}).

\begin{lem}\label{L:2homo}
The above diagram \eqref{E:ori-double} is commutative and it induces an isomorphism $H_i(\Gamma^d,\bbz)\stackrel{\cong}{\longrightarrow} \bbH_i(\Gamma,\bbz)$ for $i=0,1$.
\end{lem}
\begin{proof}
This is straightforward to check and is left to the reader.
\end{proof}

\subsection{The affine semigroup ring $R(\Gamma,\phi)$ and its associated toric variety $X_{(\Gamma,\phi)}$}\label{S:Aff-smgrp}

In this section we review the definitions of the ring $R(\Gamma,\phi)$ from \cite[\S 4]{CMKVa}.
Let  $(\Gamma,\phi)$ be a graph with a totally cyclic orientation.
Consider the pointed full-dimensional rational polyhedral cone
\begin{equation}\label{E:main-cone}
\sigma_\Gamma(\phi):= \bigcap_{e\in E(\Gamma)}\{(\cdot, \phi(e)) \ge 0\}
\subset H_1(\Gamma,\bbz)\otimes_{\bbz}\bbR.
\end{equation}
(This was denoted $\sigma(\emptyset, \phi)$ in \cite[\S 3]{CMKVa}.)
According to Gordan's Lemma
(e.g.~\cite[Prop.~1.2.17]{CLS}),
 the semigroup
\begin{equation}\label{Eqn: Semgrp}
C_\Gamma(\phi):=\sigma_\Gamma(\phi)\cap H_1(\Gamma,\bbz)\subset H_1(\Gamma,\bbz)=\bbz^{b_1(\Gamma)}
\end{equation}
is a positive, normal, affine semigroup, i.e.~a finitely generated subgroup isomorphic to a subsemigroup of $\bbz^d$ for some
$d\in \bbN$, such that $0$ is the unique invertible element and such that if $m\cdot z\in C_\Gamma(\phi)$ for
some $m\in \bbN$ and
$z\in \bbz^d$, then $z\in C_\Gamma(\phi)$.

Recall (\cite[Def.~4.2]{CMKVa}) that we  define
$$
R(\Gamma,\phi):=k[C_\Gamma(\phi)]
$$
 to be the affine semigroup ring associated to $C_\Gamma(\phi)$; i.e.~the $k$-algebra whose underlying vector space has basis $\{X^c\: :\: c\in C_\Gamma(\phi)\}$
and whose multiplication is defined by $X^c\cdot X^{c'}:=X^{c+c'}$.
$R(\Gamma,\phi)$ is a normal, Cohen--Macaulay domain of dimension equal to (e.g.~\cite[Lem.~4.3]{CMKVa})
 \begin{equation}\label{eqndim}
\dim R(\Gamma,\phi)=\dim \sigma_\Gamma(\phi)= b_1(\Gamma).
\end{equation}
The affine variety
\begin{equation}\label{E:X-orient}
X_{(\Gamma,\phi)}:=\Spec R(\Gamma,\phi)
\end{equation}
is the toric variety associated to the fan $\Sigma_{(\Gamma,\phi)}$ consisting of the dual cone $\sigma_{\Gamma}(\phi)^{\vee}\subset H_1(\Gamma,\bbz)^{\vee}\otimes_{\bbz}\bbR$ together with all its faces.

\section{The cographic toric variety $X_{\Gamma}$ and the cographic toric ring $U(\Gamma)$}\label{S:cogtoric}

Fix a graph $\Gamma$. Using the notation of \S\ref{S:HomGra}, set $\mathbb M_{\Gamma}:=\mathbb H_1(\Gamma,\bbz)$ and $\mathbb N_{\Gamma}:=\mathbb H_1(\Gamma,\bbz)^{\vee}$.
Consider the pointed rational polyhedral cone
\begin{equation}\label{E:cone+}
\sigma_\Gamma:=\bigcap_{\er \in \vec E}\{(, \er)\ge 0\}\subset \mathbb M_\Gamma\otimes _{\mathbb Z}\mathbb R,
\end{equation}
and denote by $\sigma_\Gamma^\vee\subset \bbN_\Gamma\otimes _{\mathbb Z}\mathbb R$ its dual cone.
 Again from  Gordan's Lemma, the semigroup
\begin{equation}\label{E:smgrC}
C(\Gamma):=\mathbb H_1(\Gamma,\mathbb Z)\cap \sigma_{\Gamma}
\end{equation}
is a positive, normal, affine semigroup.

\begin{defi}\label{D:cogringvar}
\noindent
\begin{enumerate}[(i)]
\item The \emph{cographic toric ring}Ê $U(\Gamma)$ of $\Gamma$ (over a base field $k$) is the affine semigroup $k$-algebra  associated to $C(\Gamma)$, i.e.
$$U(\Gamma):=k[C(\Gamma)].$$
Explicitly, $U(\Gamma)$ is the $k$-algebra whose underlying vector space has basis $\{X^c\: :\: c\in C(\Gamma)\}$ and whose multiplication is defined by $X^c\cdot X^{c'}:=X^{c+c'}$.

\item The \emph{cographic toric variety} Ê$X_{\Gamma}$ of $\Gamma$ (over a base field $k$) is the affine variety
$$
X_\Gamma:=\operatorname{Spec} U(\Gamma) =\Spec k[C(\Gamma)].
$$
\end{enumerate}
\end{defi}
Observe that $X_{\Gamma}$ is the (normal) toric variety associated to the rational polyhedral fan $\Sigma_\Gamma$ in $\bbN_{\Gamma}\otimes_{\bbz} \bbR$ formed by $\sigma_\Gamma^{\vee}$ and all its faces.  We describe $\sigma_\Gamma^\vee$ in more detail in \S \ref{S:sing}.

\begin{exa}\label{E:loop1}
Let $L$ be the loop graph, i.e.~the graph with one vertex $v$ and one unoriented edge $e$ which is a loop around $v$.
Then $\bbC_1(L,\bbz)$ is freely generated by $\er$ and $\el$ and the boundary map $\bbD$ is trivial;
hence $\bbH_1(L,\bbz)=\bbC_1(L,\bbz)=\langle \er, \el\rangle$. The cone $\sigma_{L}$ of \eqref{E:cone+} is the first quadrant in $\bbH_1(L,\bbz)\otimes_{\bbz} \bbR\cong \bbR^2$ and the semigroup
$C(L)$ of \eqref{E:smgrC} is isomorphic to $\bbN^2$, generated by $\er$ and $\el$. Therefore
$$
U(L)=k[C(L)]=k[X^{\er},X^{\el}]\cong k[X,Y] \: \text{ and } \: X_L=\Spec U(L)=\bbA_k^2.
$$
\end{exa}

The cographic toric ring $U(\Gamma)$ and the cographic toric variety $X_{\Gamma}$ admit also another presentation in terms of the affine semigroup algebra (and its corresponding affine toric variety)
associated to the double graph $\Gamma^d$ with its double orientation $\Gamma^d$, see \S\ref{S:double} and \S\ref{S:Aff-smgrp}.

\begin{pro}\label{procan}
There is an isomorphism of $k$-algebras
\begin{equation*}
U(\Gamma)\cong R(\Gamma^d,\phi^d).
\end{equation*}
inducing the isomorphism $X_{\Gamma}\cong X_{(\Gamma^d,\phi^d)}$ of toric varieties.
\end{pro}

\begin{proof}
Comparing \eqref{E:main-cone} with \eqref{E:cone+}, it is easily checked that the isomorphism $H_1(\Gamma^d,\bbz)\stackrel{\cong}{\longrightarrow}\bbH_1(\Gamma,\bbz)$ of Lemma \ref{L:2homo} sends the cone $\sigma_{\Gamma^d}(\phi^d)$ isomorphically into the cone $\sigma_\Gamma$, and hence the semigroup $C_{\Gamma^d}(\phi^d)$ isomorphically onto the semigroup $C(\Gamma)$.  By taking the associated semigroup algebras we get the isomorphism $R(\Gamma^d,\phi^d)\cong U(\Gamma)$ and, by passing to prime spectra, we obtain that $X_{(\Gamma^d,\phi^d)}\cong X_{\Gamma}$.
\end{proof}

\section{An explicit presentation of the cographic toric ring $U(\Gamma)$} \label{S:X-without}

The aim of this section is to give an explicit presentation of the cographic toric ring $U(\Gamma)$, which also shows that $U(\Gamma)$ is a deformation of the cographic toric face ring $R(\Gamma)$ introduced
and studied in \cite{CMKVa}.

To begin, we will define a  map
$$
\psi:H_1(\Gamma,\mathbb Z)\times H_1(\Gamma,\mathbb Z)\to \mathbb Z_{\ge 0}^{E(\Gamma)}.
$$
For a cycle $z\in H_1(\Gamma,\mathbb Z)\subseteq C_1(\Gamma,\mathbb Z)$, denote by $\Supp(z)$ (\emph{support}Ê of $z$) the set of edges of $E(\Gamma)$ that appear with non-zero coefficient
in $z$. Then we can write $z$ uniquely as
$$
z=\sum_{e\in \operatorname{Supp}(z)}a_e  [\er]
$$
with $a_e>0$ for all $e\in \operatorname{Supp}(z)$.

Now if
$$
z^{(1)}=\sum_{e\in \operatorname{Supp}(z^{(1)})}a_e^{(1)}[\er^{(1)}]\ \ \text {and } \ \ z^{(2)}=\sum_{e\in \operatorname{Supp}(z^{(2)})}a_e^{(2)}[\er^{(2)}]
$$
then define
\begin{equation}\label{E:psi}
\psi(z^{(1)},z^{(2)})_e:=\left\{
\begin{array}{ll}
0& \text{if }Êe\notin \operatorname{Supp}(z^{(1)})\cap\operatorname{Supp}(z^{(2)}),\\
0 & \text{Êif }Ê[\er^{(1)}]=[\er^{(2)}],\\
\min (a_e^{(1)},a_e^{(2)})& \text{Êif }Ê[\er^{(1)}]= -[\er^{(2)}].
\end{array}\right.
\end{equation}

\begin{rem}\label{remPsiDef}
While the definition above is made independent of an orientation, and will be useful for the proof of the theorem below, the definition may be more transparent with the introduction of an orientation.
So, for the sake of exposition, choose an orientation $\phi$ of $\Gamma$.   Then a cycle $z\in H_1(\Gamma,\mathbb Z)$ has a unique expression of the form
$
z=\sum_{e\in E}\alpha_e\phi(e),
$
with the $\alpha_e\in \mathbb Z$.
Now if
$
z^{(1)}=\sum_{e\in E}\alpha_e^{(1)}\phi(e) \ \ \text {and } \ \ z^{(2)}=\sum_{e\in E}\alpha_e^{(2)}\phi(e)
$
then define
$$
\psi(z^{(1)},z^{(2)})_e=\left\{
\begin{array}{ll}
0 & \text{Êif }Ê\alpha_e^{(1)}\alpha_e^{(2)}\ge 0,\\
\min (|\alpha_e^{(1)}|,|\alpha_e^{(2)}|)& \text{otherwise.}
\end{array}\right.
$$
In other words, we are just tallying the number of edges (with multiplicity)  that two cycles share in opposite directions.
Note that this definition agrees with the one above, and does not depend on the choice of $\phi$.
\end{rem}

\begin{rem}\label{remphi}
It follows from \cite[Cor. 3.4]{CMKVa} that  $\psi(z^{(1)},z^{(2)})=0$ if and only if $z^{(1)}$ and $z^{(2)}$ lie in a common  cone of the cographic fan $\mathcal{F}_{\Gamma}^{\perp}$
on $H_1(\Gamma,\bbz)\otimes \bbR$ (see \cite[\S 1.4]{CMKVa} and the references therein).
\end{rem}

The key to obtaining an  explicit presentation of  the cographic toric ring $U(\Gamma)$ is the following alternative description of the semigroup $C(\Gamma)$ of \eqref{E:smgrC}.

\begin{pro}\label{P:explC}
The semigroup $C(\Gamma)$ is isomorphic to the set $H_1(\Gamma,\mathbb Z)\times \mathbb Z_{\ge 0}^{E(\Gamma)}$ endowed with the structure of semigroup given by
\begin{equation}\label{E:explsmgrp}
(z_1,n_1)\times (z_2,n_2)\mapsto (z_1+z_2,\psi(z_1,z_2)+n_1+n_2).
\end{equation}
\end{pro}

In order to prove the above proposition, we will need the following two  lemmas.

\begin{lem}
Under the natural surjection $\bbH_1(\Gamma,\bbz)\twoheadrightarrow H_1(\Gamma, \bbz)$ induced by \eqref{E:diag-hom}, the semigroup $C(\Gamma)$ maps surjectively onto $H_1(\Gamma,\bbz)$.
\end{lem}
\begin{proof} We will prove this by constructing a section
\begin{equation}\label{eqnsection}
H_1(\Gamma,\mathbb Z)\to C(\Gamma).
\end{equation}
Any cycle $z\in H_1(\Gamma,\mathbb Z)$ can be written uniquely in the form
$z=\sum_{e\in \operatorname{Supp}(z)}a_e[\er]$ with $a_e>0$.   Thus
$$
z\mapsto \sum_{e\in \operatorname{Supp}(z)}a_e\er
$$
gives a well defined map $H_1(\Gamma,\mathbb Z)\to C(\Gamma)$.  It is clearly a section.
\end{proof}

\begin{lem}\label{L:setC}
$C(\Gamma)$ is the sub-semigroup of $\mathbb H_1(\Gamma,\mathbb Z)$ generated by $\{\er +\el\}_{e\in E(\Gamma)}$ and the image of the section $H_1(\Gamma,\mathbb Z)\to C(\Gamma)$ defined in \eqref{eqnsection} above.
\end{lem}
\begin{proof}
Clearly both $\{\er +\el\}_{e\in E(\Gamma)}$ as well as  the image of the section $H_1(\Gamma,\mathbb Z)\to C(\Gamma)$   lie in $C(\Gamma)$.

Now let $z\in C(\Gamma)$.  Recall that by definition this means that  $z=\sum_{\er \in \OE} a_{\er}\er$ with $a_{\er}\ge 0$ for all $\er \in \OE$.
Let $z'$ be the image of $z$ in $H_1(\Gamma,\mathbb Z)$ and let $z''$ be the image of $z'$ in $C(\Gamma)$ under the section.  Then $z-z''\in \ker\left(\mathbb H_1(\Gamma,\mathbb Z)\to H_1(\Gamma,\mathbb Z)\right)$.   Thus, using Lemma \ref{L:kerH1}, we can write $z=z''+\sum_{e\in E} b_e(\er +\el)$.  But by  the construction of $z''$, for all $e\in E$, the coefficient of either $\er$ or $\el$ in $z''$ is zero.  Thus $b_e\ge 0$ for all $e\in E$, and we are done.
\end{proof}

\begin{proof}[Proof of Proposition \ref{P:explC}]
It follows from Lemma \ref{L:setC} that there is an explicit bijection between the sets $C(\Gamma)$ and $H_1(\Gamma,\mathbb Z)\times \mathbb Z_{\ge 0}^{E(\Gamma)}$. By tracing the semigroup law on $C(\Gamma)$ via this bijection we ended up exactly with the  semigroup law on $H_1(\Gamma,\mathbb Z)\times \mathbb Z_{\ge 0}^{E(\Gamma)}$ given by \eqref{E:explsmgrp}, and we are done.
\end{proof}

From the explicit description of the semigroup $C(\Gamma)$ given in Proposition \ref{P:explC}, we derive the following explicit presentation of the cographic toric ring $U(\Gamma)$.

\begin{teo}\label{T:exp-pres}
Let $\Gamma$ be a connected graph.
Consider the $k$-algebra $D(\Gamma)$  whose underlying vector space has basis $\{X^zT^{\epsilon}\: :\: z\in H_1(\Gamma,\mathbb Z), \epsilon \in \mathbb Z_{\ge 0}^{E(\Gamma)}\}$
and whose multiplication is defined by  the rule $X^zT^{\epsilon}\cdot X^{z'}T^{\epsilon'}:=X^{z+z'}T^{\psi(z,z')+\epsilon+\epsilon'}$.
In other words,
$$
D(\Gamma):= \frac{k[X^z]_{z\in H_1(\Gamma,\mathbb Z)}[T_e]_{e\in E}}{(X^zX^{z'}-X^{z+z'}\vec T^{\psi(z,z')})}.
$$
Then we have an isomorphism $U(\Gamma)\cong D(\Gamma)$.
\end{teo}
\begin{proof}
Observe that $D(\Gamma)$ is the semigroup $k$-algebra associated to the set  $H_1(\Gamma,\mathbb Z)\times \mathbb Z_{\ge 0}^{E(\Gamma)}$ endowed with semigroup law \eqref{E:explsmgrp}.
Then the result follows from Proposition \ref{P:explC}.
\end{proof}

\begin{rem}\label{R:defcog1}
From Theorem \ref{T:exp-pres} together with Remark \ref{remphi}, it follows that by setting the variables $T_e$ equal to zero we get a surjective morphism of $k$-algebras
\begin{equation}\label{E:quotmor}
U(\Gamma)\cong D(\Gamma)\twoheadrightarrow R(\Gamma)
\end{equation}
where $R(\Gamma)$ is the cographic toric face ring introduced in \cite[Def.~1.2]{CMKVa}.
Thus the cographic toric variety $X_\Gamma$ can be viewed as a deformation of the cographic toric face variety  $\operatorname{Spec}R(\Gamma)$ (\cite[Def.~1.2]{CMKVa}) over the base
$\operatorname{Spec}k[T_e]_{e\in E}$.
\end{rem}

\section{Singularities of $X_\Gamma$}\label{S:sing}

The aim of this section is to study the singularities of the cographic toric variety $X_{\Gamma}$.
Recall from Definition \ref{D:cogringvar}  that $X_{\Gamma}$ is the (normal) toric variety associated to the rational polyhedral fan $\Sigma_\Gamma$ in $\bbN_{\Gamma}\otimes_{\bbz} \bbR$ formed by the
rational polyhedral cone $\sigma_\Gamma^{\vee}$ \eqref{E:cone+} and all its faces.
The following lemma summarizes the basic properties of the cone $\sigma_{\Gamma}^\vee$.

\begin{lem}\label{L:prop-cone}
Let $\Gamma$ be a connected graph.  Set $\mathbb M_\Gamma=\mathbb H_1(\Gamma,\mathbb Z)$.  

\begin{enumerate}[(i)]

\item \label{L:prop-cone0}  The cone $\sigma_{\Gamma}^\vee\subseteq \bbM_{\Gamma}^\vee=\bbN_{\Gamma}$ is equal to
$$\sigma_{\Gamma}^\vee=\left\{\sum_{\er\in \OE(\Gamma)}a_{e}\cdot (\ , \er)\: : \: a_{e}\geq 0. \right\},$$
where $(\ ,\er)$ denotes the element of $\bbN_{\Gamma}$ obtained by pairing an element of $\bbM_\Gamma$ with $\er$ via the scalar product $(\ ,\ )$ defined in \S\ref{S:HomGra}.

\item \label{L:prop-cone1} The cone $\sigma_{\Gamma}^\vee$ is pointed and of full dimension in $\bbN_{\Gamma}\otimes \mathbb R$.

\item \label{L:prop-cone2} The extremal rays  of $\sigma_{\Gamma}^\vee$ are of the form $\langle (\ , \er) \rangle:= \bbR_{\geq 0}\cdot (\ , \er) $ as $\er$ varies in  $\OE(\Gamma)$. Moreover, given
 $\er_1\neq \er_2$, we have that
 $$ \langle (\ , \er_1) \rangle = \langle (\ , \er_2) \rangle \Longleftrightarrow \un \er_1=\un \er_2 \: \text{ is a separating edge of } \Gamma. $$

\item \label{L:prop-cone3} For every $\er\in \OE(\Gamma)$, the primitive element of the ray $\langle (\ , \er)\rangle$ with respect to the lattice $\bbN_{\Gamma}$ is $(\ , \er)$, i.e.
$\langle (\ , \er)\rangle\cap \bbN_{\Gamma}=\bbZ_{\geq 0} \cdot (\ , \er)$.

\end{enumerate}
\end{lem}

\begin{proof}
\eqref{L:prop-cone0} follows from \eqref{E:cone+}, using  the definition  of a dual cone. Also,  the first part of \eqref{L:prop-cone2}  follows from \eqref{L:prop-cone1}.
We will deduce the remaining properties of $\sigma_{\Gamma}^\vee$ from the properties of its dual cone $\sigma_{\Gamma}\subset \bbM_{\Gamma}$, which is isomorphic to the cone
$\sigma_{\Gamma^d}(\sigma^d)\subset H_1(\Gamma^d,\bbz)$ as shown in the proof of Proposition \ref{procan}.

According to \cite[Prop.~3.1]{CMKVa}, the cone $\sigma_{\Gamma^d}(\phi^d)$ is a pointed and full-dimensional cone in
$H_1(\Gamma^d,\bbz)\otimes \bbR$.
By duality, we deduce that \eqref{L:prop-cone1} holds.

Observe now that if $e$ is a non-separating edge of $\Gamma$, then the orientation $\phi^d_{|\Gamma^d\setminus \{e'\}}$ (resp. $\phi^d_{|\Gamma^d\setminus \{e''\}}$) induced by $\phi^d$ on the graph $\Gamma^d\setminus  \{e'\}$
(resp. $\Gamma^d\setminus  \{e''\}$)  is still totally cyclic. On the other hand, if $e$ is a separating edge of $\Gamma$ then the corresponding edges $e'$ and $e''$ of $\Gamma^d$ form a pair of parallel edges; hence the orientation $\phi^d_{|\Gamma^d\setminus \{e', e''\}}$ induced by $\phi^d$ on $\Gamma^d\setminus  \{e', e''\}$ is still totally cyclic, while neither the orientation induced by $\phi^d$  on $\Gamma^d\setminus  \{e'\}$
nor the one induced on $\Gamma^d\setminus  \{e''\}$ is totally cyclic.  
Therefore, \cite[Prop.~3.1]{CMKVa} implies that the codimension one faces of $\sigma_{\Gamma^d}(\phi^d)$ are given by (with the notation of 
\S\ref{S:Aff-smgrp}) 
\begin{enumerate}[(i)]
\item \label{E:face1} $\sigma_{\Gamma^d\setminus\{e'\}}\left(\phi^d_{|\Gamma^d\setminus \{e'\}}\right)$ and $\sigma_{\Gamma^d\setminus\{e''\}}\left(\phi^d_{|\Gamma^d\setminus \{e''\}}\right)$ \: for any non-separating edge $e$ of 
$\Gamma$;
\item \label{E:face2} $\sigma_{\Gamma^d\setminus\{e', e''\}}\left(\phi^d_{|\Gamma^d\setminus \{e', e''\}}\right)$  \: for any separating edge $e$ of  $\Gamma$.
\end{enumerate}
The faces of type \eqref{E:face1} are given by  intersecting $\sigma_{\Gamma^d}(\phi^d)$ with, respectively, the hyperplanes $\{(\cdot, \phi^d(e'))=0\}$ and $\{(\cdot, \phi^d(e''))=0\}$ for any non-separating edge $e$ of $\Gamma$; on the other hand, the faces in \eqref{E:face2} are given by intersecting with the hyperplanes $\{(\cdot, \phi^d(e'))=0\}=\{(\cdot, \phi^d(e''))=0\}$ for any separating edge $e$ of  $\Gamma$.
By duality, we obtain  \eqref{L:prop-cone2}.

Part \eqref{L:prop-cone3}: consider the element $\er+\el\in \bbH_1(\Gamma,\bbZ)=\bbM_{\Gamma}$. Since $(\er+\el, \er)=1$, we get that $(\ ,\er)$ is the primitive element of the ray $\langle (\ , \er)\rangle$. 
\end{proof}

\begin{cor} \label{CorDim}
$X_{\Gamma}$  does not contain torus factors and it has dimension equal to
\begin{equation}\label{E:dim-bbH}
\dim  X_{\Gamma}=\dim \sigma_{\Gamma}^\vee =\rank \bbH_1(\Gamma,\bbz)=b_1(\Gamma)+|E(\Gamma)|.
\end{equation}

\end{cor}

\begin{proof}
This  follows directly from Lemma \ref{L:prop-cone} \eqref{L:prop-cone1}, using \cite[Prop.~3.3.9 (c)]{CLS} and \cite[Prop.~6.6.1]{BH}.
\end{proof}

We will want  the following result describing the behavior of the
cographic toric variety in the presence of separating edges and loops.

\begin{lem}\label{L:sep-edge}
Let $\Gamma$ be a connected graph with $n$ separating edges and $m$ loops, and let $\Gamma'$ be the graph obtained from $\Gamma$ by contracting the separating edges and deleting the loops.
Then we have that
$$
X_{\Gamma}= \mathbb A_k^{n+2m}\times X_{\Gamma'}.
$$
\end{lem}
\begin{proof}
Let $\{f_1,\ldots,f_n\}$ be the separating edges of $\Gamma$, $\{e_1,\ldots,e_m\}$ the loops of $\Gamma$ and set $\gamma_i:=[\er_i]\in H_1(\Gamma,\bbz)$.
Clearly we have that $H_1(\Gamma,\bbz)=H_1(\Gamma',\bbz)\oplus \bigoplus_{i=1}^m \langle \gamma_i\rangle$. Moreover, if we denote by $\psi$ the map \eqref{E:psi} associated to $H_1(\Gamma,\bbz)$
and by $\psi'$ the analogous map associated to $H_1(\Gamma',\bbz)$, then we have that
\begin{equation*}
\psi\left(z^{(1)}+\sum_i n_i^{(1)} \gamma_i \ , \  z^{(2)}+\sum_i n_i^{(2)}\gamma_i\right)_{e}=Ê
\begin{cases}
\psi'(z^{(1)}, z^{(2)}) & \text{Êif }Êe\not\in \{e_1,\ldots,e_m\},\\
0 & \text{ if } e=e_i \text{ and }Ên_i^{(1)}n_i^{(2)}\geq 0, \\
\min (|n_i^{(1)}|,|n_i^{(2)}|) &  \text{ if } e=e_i \text{ and }Ên_i^{(1)}n_i^{(2)}< 0, \\
\end{cases}
\end{equation*}
for any $z^{(j)}\in H_1(\Gamma',\bbz)$ and $n_j^{(i)}\in \bbz$. This implies easily that (using the notation of Theorem \ref{T:exp-pres})
$$D(\Gamma)=D(\Gamma')\otimes_k \frac{k[X^{\gamma_1},X^{-\gamma_1},\ldots, X^{\gamma_{m}}, X^{-\gamma_{m}}, T_{e_1},\ldots,T_{e_m}]}{\left(X^{\gamma_1} X^{-\gamma_1}-T_{e_1},\ldots,
X^{\gamma_{m}}X^{-\gamma_{m}}-T_{e_m}\right)}\otimes_k k[T_{f_1},\ldots, T_{f_n}].$$
By passing to the prime spectra and using Theorem \ref{T:exp-pres}, we conclude.
\end{proof}

\begin{rem}
A lengthier, but more elementary argument can be made for Lemma \ref{L:sep-edge} directly from the definitions, without using Theorem \ref{T:exp-pres}.
\end{rem}

From the point of view of birational geometry, the singularities of $X_{\Gamma}$ are particularly nice:

\begin{teo}\label{T:sing}
The variety $X_\Gamma$ is Gorenstein, terminal and  has  rational singularities.
\end{teo}
\begin{proof}
It is well-known that any (normal)  toric variety has rational singularities (e.g.~\cite[Thm.~11.4.2]{CLS}) and  is Cohen--Macaulay
(e.g.~\cite[Thm.~9.2.9]{CLS}).

According to   \cite[Prop.~8.2.12]{CLS} (see also Proposition \ref{proGC}), $X_{\Gamma}$ is Gorenstein, i.e.~the canonical divisor $K_{X_{\Gamma}}$ is Cartier, if and only if there exists an element
$m\in \mathbb M_\Gamma$ such that $\langle m, u_{\rho}\rangle=1$ for any extremal ray $\rho$ of $\sigma_{\Gamma}^\vee$, where
$\langle\ ,\  \rangle$ denotes the canonical pairing between $\mathbb M_\Gamma$ and $\mathbb N_\Gamma=\mathbb M_\Gamma^{\vee}$ and $u_{\rho}$ denotes the minimal generator of $\rho\cap \mathbb N_\Gamma$. Consider now the following element of $\mathbb C_1(\Gamma,\bbz)$
\begin{equation}\label{E:ele-m}
m_\Gamma:=\sum_{\er\in \OE} \er=\sum_{e\in E}(\er+\el).
\end{equation}
Since $\mathbb D(\el)=-\mathbb D(\er)$, we get that $\mathbb D(m_\Gamma)=0$, and hence that $m_\Gamma\in
\mathbb M_\Gamma=\bbH_1(\Gamma,\bbz)$. By definition of the scalar product (see \S\ref{S:HomGra}), we easily get that
\begin{equation*}
(m_\Gamma,\er)=1 \text{ for any } \er \in \OE.
\end{equation*}
For brevity, we will use the notation  $u_{\er}$ for the element $(\ ,\er) \in \mathbb N_\Gamma$ determined by  $\er \in \OE$. The above equality translates
into
\begin{equation}\label{E:m-rays}
\langle m_\Gamma,u_{\er}\rangle=1.
\end{equation}
By Lemma \ref{L:prop-cone}, we get that the rays of $\sigma_{\Gamma}^\vee$ are all of the form $\langle u_{\er} \rangle=\bbR_{\geq 0}\cdot u_{\er}$ (as $\er$ varies in $\OE$) and that  $u_{\er}$ is the primitive element of the ray 
$\langle u_{\er}\rangle$. Therefore we conclude that $X_{\Gamma}$ is Gorenstein.

Finally, let us show that $X_{\Gamma}$ has terminal singularities. Since we have already proved that $X_{\Gamma}$ is
Gorenstein, we conclude that $X_{\Gamma}$ has canonical singularities by \cite[Prop.~11.4.11]{CLS}. Thus, using
\cite[Prop.~11.4.12]{CLS} (see also Proposition \ref{proCC}), we conclude that in order to prove that $X_{\Gamma}$ has terminal singularities
it is (necessary and) sufficient to prove the following:

\un{CLAIM}: \emph{If $x\in \sigma_{\Gamma}^\vee\cap \mathbb N_\Gamma$ is such that
$\langle m_\Gamma, x\rangle=1$, then $x=u_{\er}$ for some $\er\in \OE$}.

By Lemma \ref{L:prop-cone} \eqref{L:prop-cone0},  we can write $x=\sum_{\er \in \OE } a_{\er}\cdot
u_{\er}$ for certain $a_{\er}\in \bbR_{\geq 0}$. Note that such a representation may not be unique if the cone
$\sigma_{\Gamma}^\vee$ is not simplicial, but we fix one such representation. By hypothesis, and recalling the definition of $m_\Gamma$ 
\eqref{E:ele-m}, we have that
\begin{equation}\label{E:1sum}
1=\langle m_\Gamma, x\rangle=\left\langle \sum_{\er '\in \OE} \er', \sum_{\er\in \OE}
a_{\er}\cdot u_{\er}\right\rangle=\sum_{\er\in \OE} a_{\er}.
\end{equation}
Consider now, for any $e\in E(\Gamma)$, the element $\gamma_e:=\er+\el\in \bbC_1(\Gamma,\bbz)$.
As above, since $\mathbb D(\el)=-\mathbb D(\er)$, we get that $\mathbb D(\gamma_e)=0$; i.e.~that
$\gamma_e\in \mathbb M_\Gamma=\bbH_1(\Gamma,\bbz)$. Using \eqref{E:1sum} and the fact that $a_{\er}\geq 0$, we get that
\begin{equation*}
\langle \gamma_e, x\rangle = a_{\er}+a_{\el} \in [0,1].
\end{equation*}
Moreover, since $x\in \mathbb N_\Gamma$ and $\gamma_e\in \mathbb M_\Gamma$, we get that $\langle \gamma_e, x\rangle\in \bbz$;
hence $\langle \gamma_e, x\rangle$ is equal either to $1$ or to $0$. In the first case, all the coefficients
$a_{\er}$ with $\er\neq \er$ or $\el$ must vanishes because of \eqref{E:1sum}; hence
$x=a_{\er} u_{\er}+ a_{\el} u_{\el}$. In the second case, i.e.~if $\langle \gamma_e, x\rangle=0$, then necessarily
$a_{\er}=a_{\el}=0$. We can therefore iterate the argument using all the  edges of $\Gamma$ and,
since $x\neq 0$, in the end we find that necessarily
\begin{equation}\label{E:2sum}
x=a_{\er} u_{\er}+ a_{\el} u_{\el} \: \text{ for some } e\in E(\Gamma).
\end{equation}
By virtue of  Lemma \ref{L:sep-edge} we may assume  that $\Gamma$ does not have separating edges, so in particular $e$ is not a separating
edge of $\Gamma$. Using this, it is easy to see  that there exists a cycle $\gamma\in \bbH_1(\Gamma, \bbz)$ that contains $\er$ but not $\el$.
Therefore, from \eqref{E:1sum} and \eqref{E:2sum}, we get that
$$\langle \gamma, x\rangle = a_{\er} \in [0,1].$$
However, since $x\in \mathbb N_\Gamma$ and $\gamma\in \mathbb M_\Gamma$, we get that $\langle \gamma, x\rangle\in \bbz$;
hence $a_{\er}=\langle \gamma, x\rangle$ is equal either to $1$ or to $0$, which implies that $x$ is equal either
to $u_{\er}$ or to $u_{\el}$; the claim is now proved.
\end{proof}

We can now give a  complete classification of the graphs $\Gamma$ for which $X_\Gamma$ is smooth or  has finite quotient singularities. 
\begin{pro}\label{P:smooth}
Let $\Gamma$ be a connected graph. The following conditions are equivalent:
\begin{enumerate}[(i)]
\item \label{P:smooth0} $X_\Gamma = \bbA_k^{b_1(\Gamma)+|E(\Gamma)|}$;
\item \label{P:smooth1} $X_\Gamma$ is smooth;
\item \label{P:smooth2} $X_\Gamma$ has finite quotient singularities;
\item \label{P:smooth3} $\Gamma$ is tree-like, i.e.~$\Gamma$ becomes a tree after removing all the loops.
\end{enumerate}
\end{pro}
\begin{proof}

\eqref{P:smooth3} $\Rightarrow$ \eqref{P:smooth0} follows from Lemma \ref{L:sep-edge}.

\eqref{P:smooth0} $\Rightarrow$ \eqref{P:smooth1} $\Rightarrow$ \eqref{P:smooth2} are obvious.

\eqref{P:smooth2} $\Leftrightarrow$ \eqref{P:smooth3}: First of all, from Lemma \ref{L:sep-edge} we get that it is enough to prove the statement under the hypothesis that $\Gamma$ has no separating edges.  Note that, under this assumption, condition \eqref{P:smooth3} now simply becomes that $\Gamma$ has a unique vertex.
According to \cite[Thm.~11.4.8]{CLS}, $X_{\Gamma}$ has finite quotient singularities if and only if $\sigma_{\Gamma}^\vee$ is simplicial, i.e.~the number of its extremal rays is
equal to its dimension. By Lemma \ref{L:prop-cone}\eqref{L:prop-cone1}, the dimension of $\sigma_{\Gamma}^\vee$ is equal to
$$\dim \sigma_{\Gamma}^\vee=\dim \bbH_1(\Gamma,\bbz)=b_1(\Gamma)+|E(\Gamma)|=2|E(\Gamma)|-|V(\Gamma)|+1. $$
and the number of its extremal rays is equal to $2|E(\Gamma)|$ by Lemma \ref{L:prop-cone}\eqref{L:prop-cone2}.
Therefore $\sigma_{\Gamma}^\vee$ is simplicial if and only if $\Gamma$ has a unique vertex, and we are done.
\end{proof}

Denote by $\un 0$ the unique torus fixed point of the affine toric variety $X_{\Gamma}$ and let $\m$ the maximal ideal of $U(\Gamma)$ corresponding to it.
Explicitly, under the isomorphism $U(\Gamma)\cong D(\Gamma)$ of Theorem \ref{T:exp-pres}, the ideal $\m$ is generated by the variables $X^z$ and the variables $T_e$.
The dimension of the tangent space of $X_\Gamma$ at $\un{0}$, or equivalently the embedded dimension of $U(\Gamma)$ at $\m$, is easy to determine in terms of  the (unoriented) circuits $\Cir(\Gamma)$  of $\Gamma$  and of the loops $\Loops(\Gamma)$ of $\Gamma$.

\begin{pro}\label{protan}
The zariski tangent space $T_{\un 0}(X_\Gamma)$ at $\un{0}$ has dimension equal to
$$\dim T_{\un 0}(X_\Gamma)= 2 |\Cir(\Gamma)|+|E(\Gamma)|-|\Loops(\Gamma)|.$$
\end{pro}
\begin{proof}
By \cite[Thm.~4.15(i), Prop.~5.2]{CMKVa}, the embedded dimension of $R(\Gamma^d,\phi^d)$ at $\m$ is equal
to the cardinality of the set $\Cir_{\phi^d}(\Gamma^d)$ of oriented circuits compatibly oriented with $\phi^d$. Therefore, we conclude by applying Proposition \ref{procan} and the lemma below.
\end{proof}

\begin{lem}\label{L:or-circ}
The set $\Cir_{\phi^d}(\Gamma^d)$ of oriented circuits compatibly oriented with
$\phi^d$ is equal to
$$|\Cir_{\phi^d}(\Gamma^d)|=2|\Cir(\Gamma)|+|E(\Gamma)|-|\Loops(\Gamma)|
$$
\end{lem}
\begin{proof}
For every $e\in E(\Gamma)\setminus \Loops(\Gamma)$, we set $\eta_e:=\phi^d(e')+\phi^d(e'')\in \Cir_{\phi^d}(\Gamma^d)$
where $e'$ and $e''$ are the two edges of $\Gamma^d$ corresponding to $e\in \Gamma$.
By taking the image of a circuit of $\Gamma^d$ under the natural contraction map $\Gamma^d\to \Gamma$, we get a well-defined map
$$\Cir_{\phi^d}(\Gamma^d)\setminus \{\eta_e \: :\:e \in E(\Gamma)\setminus \Loops(\Gamma)\}\to \Cir(\Gamma).$$
Since any circuit of $\Gamma$ can be lifted in exactly two ways to an oriented circuit of $\Gamma^d$ compatibly oriented with $\phi^d$, the above map is surjective and $2\!:\!1$. This concludes the proof.
\end{proof}

\begin{rem} In terms of Theorem  \ref{T:exp-pres}, Proposition \ref{protan} reflects the fact that $D(\Gamma)$ is generated by $2 |\Cir(\Gamma)|$ ``$X$'' variables and $|E(\Gamma)|$ ``$T$'' variables, and has $|\Loops(\Gamma)|$ relations involving linear terms.
\end{rem}

We now consider the multiplicity of $X_\Gamma$ at $\underline 0$.
To that aim, we need to recall some definitions. Let $ H_{\mathbb Z}$ be a lattice and let $\sigma$ be  a strongly convex rational polyhedral cone in $H_{\mathbb R}= H_{\mathbb Z}\otimes \mathbb R$.   Set $C(\sigma):=\sigma\cap H_{\mathbb Z}$, $H_{\bbR,\sigma}=\langle \sigma\rangle \subseteq H_{\bbR}$ to be the linear span of $\sigma$ in $H_{\mathbb R}$, and
$H_{\bbz,\sigma}:=\langle \sigma \rangle \cap H_{\bbz}$;
note $H_{\mathbb R,\sigma}=H_{\mathbb Z,\sigma}\otimes \mathbb R$.
We denote by $\operatorname{vol}_{C(\sigma)}$ the unique translation-invariant measure on
$H_{\mathbb R,\sigma}$ such that the volume of a standard unimodular simplex $\Delta$ is $1$ (i.e.~$\Delta$ is the convex hull of a basis of $H_{\bbz,\sigma}$ together
with  $0$). Following \cite[p.184]{GKZ}, denote by $K_+(C(\sigma))\subseteq H_{\mathbb R,\sigma}$ the convex hull of the set
$C(\sigma)\setminus \{0\}$ and by $K_-(C(\sigma))$ the closure of $\sigma\setminus K_+(C(\sigma))$. The set $K_-(C(\sigma))$ is a bounded (possibly not convex)
lattice polyhedron in $H_{\bbR,\sigma}$ which is called the {\em subdiagram part} of $C(\sigma)$.

\begin{defi}\cite[Ch.~5, Def.~3.8]{GKZ}
The {\em subdiagram volume}  of $C(\sigma)$ is the natural number
$$u(C(\sigma)):=\operatorname{vol}_{C(\sigma)}(K_-(C(\sigma))).$$
\end{defi}

Now let  $R(\sigma)=R(C(\sigma))$ be the semigroup ring associated to $C(\sigma)$.  Let  $\m$ be the maximal ideal generated by the generators of the $k$-algebra $R(\sigma)$.  Let $\underline 0$ be the corresponding point in $X_\sigma:=\operatorname{Spec}R(\sigma)$.  The multiplicity of $X_\sigma$ at $\underline 0$ is given by (see e.g.~\cite[Ch.~5, Thm.~3.14]{GKZ})
$$
\operatorname{mult}_{\underline 0}X_\sigma=u(C(\sigma)).
$$

\begin{teo}\label{T:multiGKZ}
Let $\Gamma$ be a graph, and let $\sigma=\sigma_\Gamma$.
$$\operatorname{mult}_{\underline 0}X_\Gamma=u(C(\sigma)).$$
\end{teo}

\begin{rem}
It would be interesting to have a formula for $\operatorname{mult}_{\un 0}X_\Gamma$ in terms of standard invariants of the graph $\Gamma$ (or $\Gamma^d$).
\end{rem}

\section{The cographic toric ring $U(\Gamma)$ as a ring of invariants}\label{S:U-inv}

In this section, we show that the cographic toric ring $U(\Gamma)$ appears as ring of invariants of a torus acting on a certain polynomial ring.
Indeed, this invariant ring will appear in Section \ref{S:univJac} in the description of the completed local rings of the universal compactified Jacobian.

Consider the action of the algebraic torus $T_\Gamma:=\prod_{v\in V(\Gamma)}\mathbb G_m$ on the polynomial ring
$$ B(\Gamma) := k[X_{\er}: \er \in \OE]
$$
given by the rule that $\lambda=(\lambda_v)_{v\in V(\Gamma)}\in T_\Gamma$ acts as
\begin{equation}\label{E:act-coord}
\lambda \cdot X_{\er}=\lambda_{s(\er)} \lambda_{t(\er)}^{-1} X_{\er}.
\end{equation}

\noindent In order to more easily connect the results of this paper to those in \cite{CMKVb}, we note the following.

\begin{rem}\label{R:B-action}
The ring $B(\Gamma)$ is isomorphic to
$$B(\Gamma)\cong \frac{k[X_e,Y_e,T_e: e\in E(\Gamma)]}{(X_eY_e-T_e)}$$
and its  completion at the maximal ideal $(X_e,Y_e)$ is isomorphic to the ring
denoted $\widehat{B(\Gamma)}$ in \cite[Thm.~A]{CMKVb}.
Under this isomorphism, the action of $T_\Gamma$ on $B(\Gamma)$ given above, induces the same action of $T_\Gamma$  on $\widehat{B(\Gamma)}$ given in  \cite[Thm.~A]{CMKVb}.
\end{rem}

\begin{teo}\label{T:pres-quot}
The cographic toric ring $U(\Gamma)$ is isomorphic to the subring $B(\Gamma)^{T_{\Gamma}}\subset B(\Gamma)$ of $T_{\Gamma}$-invariants on $B(\Gamma)$.
\end{teo}

\begin{proof}
Using Theorem \ref{T:exp-pres}, we are going to show that $B(\Gamma)^{T_{\Gamma}}$ is isomorphic to the $k$-algebra $D(\Gamma)$.
The proof is essentially identical to the proof of \cite[Thm.~6.1]{CMKVa}; we first show that the underlying $k$-vector spaces agree, and then we show that the multiplication rules agree.
In keeping with the notation of the proof of  \cite[Thm.~6.1]{CMKVa}, we first observe (as in Remark \ref{R:B-action}) that $B(\Gamma)$ can be identified with $$k[X_{\el},X_{\er},T_e : e \in E(\Gamma)]/(X_{\el}X_{\er}-T_e).$$  The key point is then to identify the invariant monomials in this ring.   This is made easier by the observation that every monomial has an expression of the form
$$
\prod _{e\in E(\Gamma)} X_{\er}^{a_ e}T_e^{b_e}
$$
with $a_e,b_e\in \mathbb Z_{\ge 0}$, where for each $e\in E(\Gamma)$ we have that $\er$ is one of the two orientations of $e$.  The expression is unique up to replacing $\er$ with $\el$ for those $e$ such that $a_e=0$.
  The same direct analysis of the action  as in the proof of  \cite[Thm.~6.1]{CMKVa} shows that in order for this monomial to be invariant, $\sum_{e\in E(\Gamma)}a_e\er\in H_1(\Gamma,\mathbb Z)$.    Thus as $k$-vector spaces, $D(\Gamma)$ and $B(\Gamma)^{T_\Gamma}$ agree.  It remains to check that multiplication agrees.  This can be checked at the level of monomials, and $\psi$ in the definition of $D(\Gamma)$ (see Remark \ref{remPsiDef}) was constructed exactly to make these agree.
\end{proof}

The cographic toric ring $U(\Gamma)$ is related to the cographic toric face ring $R(\Gamma)$ studied in \cite{CMKVa}, as explained in the following remark (see also Remark \ref{R:defcog1}).

\begin{rem}\label{R:A-action}
The action of $T_\Gamma$ on $B(\Gamma)$ defines an action on the quotient
$$A(\Gamma):=\frac{B(\Gamma)}{(X_{\er}X_{\el}:e\in E(\Gamma))}\cong \frac{k[X_e,Y_e: e\in E(\Gamma)]}{(X_eY_e)},$$
which coincides with the action of $T_{\Gamma}$ defined in \cite[Thm.~A]{CMKVa}.
Therefore, the natural surjection $B(\Gamma) \twoheadrightarrow A(\Gamma)$ induces, by taking $T_{\Gamma}$-invariants, a map
\begin{equation}\label{eqncogquot}
U(\Gamma)\longrightarrow R(\Gamma)
\end{equation}
where $R(\Gamma):=A(\Gamma)^{T_{\Gamma}}$ is the cographic toric face ring of $\Gamma$ (see \cite[Thm.~6.1]{CMKVa}).
Indeed, the morphism \eqref{eqncogquot} coincides with the morphism \eqref{E:quotmor} and in particular it is surjective.
\end{rem}

\section{Examples}\label{S:exa}

We now include a few examples of cographic toric rings.

\subsection{The $n$-cycle $C_n$}

Let $C_n$ be the $n$-cycle graph, i.e.~the graph formed by $n$ vertices connected by a closed chain of $n$ edges, as depicted in Figure \ref{F:cycle}.

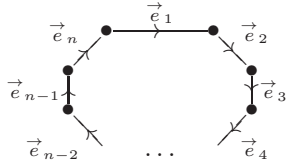
\begin{figure}[ht]
\begin{equation*}
\xymatrix@=.7pc{
& *{\bullet}  \ar@{-}[dl]|-{\SelectTips{cm}{}\object@{<}}_{\er_{n}} \ar@{-}[rr]|-{\SelectTips{cm}{}\object@{>}}^{\er_{1}}& & *{\bullet} \ar@{-}[dr]|-{\SelectTips{cm}{}\object@{>}}^{\er_{2}} & \\
 *{\bullet} \ar@{-}[d]|-{\SelectTips{cm}{}\object@{<}}_{\er_{n-1}}&& &&*{\bullet}\ar@{-}[d]|-{\SelectTips{cm}{}\object@{>}}^{\er_{3}} \\
  *{\bullet} \ar@{-}[dr]|-{\SelectTips{cm}{}\object@{<}}_{\er_{n-2}} && &&*{\bullet} \ar@{-}[dl]|-{\SelectTips{cm}{}\object@{>}}^{\er_4} \\
  && \ldots &&
}
\end{equation*}
\caption{The $n$-cycle $C_n$ with half of its oriented edges.}\label{F:cycle}
\end{figure}

The cographic toric ring of $C_n$ admits the following explicit presentation
$$
U(C_n)=\frac{k[X,Y,T_1,\ldots,T_n]}{(XY-T_1\ldots T_n)}.
$$
To see this, consider the explicit presentation of the cographic toric ring given in \S\ref{S:X-without}.
Note that the are two oriented circuits of $C_n$ giving rise to the elements $c:=[\er_1]+\cdots+[\er_n]$ and $-c$ of $H_1(\Gamma,\bbz)$.
Then, using Proposition \ref{protan}, we get that the generators of the ring $U(C_n)$ are $X=X^c$, $Y=X^{-c}$ and $T_i=T_{e_i}$ for $1\leq i\leq n$.
Since the function $\psi$ of \eqref{E:psi} on $H_1(C_n,\bbz)$ is such that $\psi(c,-c)=e_1+\cdots e_n$, we get the relation $XY-T_1\ldots T_n$.

It is easily checked that the cographic toric variety $X_{C_n}=\Spec U(C_n)$ satisfies
$$\begin{sis}
& \dim X_{C_n}=n+1,\\
& \dim T_{\un 0}(X_{C_n})=n+2, \\
& \operatorname{mult}_{\underline 0}X_{C_n}=2,
\end{sis} $$
which is of course in agreement with the formulas obtained in \S\ref{S:sing}.

\subsection{The $n$-thick edge $I_n$}

Let $I_n$ be the $n$-th thick edge graph, i.e.~the graph formed by two vertices joined by $n$ edges, as depicted in Figure \ref{F:thick}.
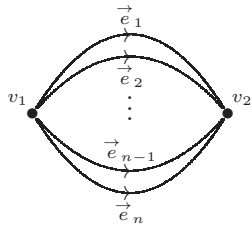
\begin{figure}[ht]
\begin{equation*}
\xymatrix@=.5pc{
 *{\bullet} \ar @{-}@/_1.8pc/[rrrrrr]|-{\SelectTips{cm}{}\object@{>}}^{\er_{n-1}} \ar @{-} @/_2.5pc/[rrrrrr] |-{\SelectTips{cm}{}\object@{>}}_{\er_n}
 \ar@{-} @/^1.8pc/[rrrrrr]|-{\SelectTips{cm}{}\object@{>}}_{\er_2}  \ar@{-}@/^2.5pc/[rrrrrr]|-{\SelectTips{cm}{}\object@{>}}^{\er_1}^<{v_1}^>{v_2}  &&& \vdots &&&*{\bullet} \\
}
\end{equation*}
\caption{The $n$-thick edge $I_n$ with half of its oriented edges.}\label{F:thick}
\end{figure}
\noindent The cographic toric ring of $I_n$ admits the following explicit presentation
$$
U(I_n)= \frac{k[X_{ij}, T_k]_{1\leq i\neq j\leq n, 1\leq k\leq n}}{(X_{ij}X_{ji}-T_iT_j, X_{ij}X_{jk}-T_j X_{ik})}.
$$
To see this, consider the explicit presentation of the cographic toric ring given in \S\ref{S:X-without}.
Note that the oriented circuits of $I_n$ gives rise to the elements  $\gamma_{ij}:=[\er_i]-[\er_j]\in H_1(I_n,\bbz)$ for any $1\leq i< j\leq n$. Then, using Proposition \ref{protan}, we deduce that the generators of the ring $U(I_n)$ are $X_{ij}=X^{\gamma_{ij}}$ for $1\leq i\neq j\leq n$ and $T_k=T_{e_k}$ for $1\leq k\leq n$.
Since the only non-zero values of the function $\psi$ of \eqref{E:psi} on the oriented circuits $\gamma_{ij}$ of $H_1(I_n,\bbz)$ are given by
$$\begin{sis}
&\psi(\gamma_{ij},\gamma_{ji})=e_i+e_j, \\
&\psi(\gamma_{ij},\gamma_{jk})=e_j, \\
\end{sis}$$
we get the desired relations among the given generators.

It is easily checked that the cographic toric variety $X_{I_n}=\Spec U(I_n)$ satisfies
$$\begin{sis}
& \dim X_{I_n}=2n-1,\\
& \dim T_{\un 0}(X_{I_n})=n^2, \\
& \operatorname{mult}_{\underline 0}X_{I_n}=\binom{2(n-1)}{n-1},
\end{sis} $$
which is of course in agreement with the formulas obtained in \S\ref{S:sing}.

\section{The universal compactified Jacobian}\label{S:univJac}

The aim of this section is to apply the results of the previous sections in order to study the singularities
of the universal compactified Jacobian $\bar J_{d,g}$ and eventually prove in Theorem \ref{T:can-sing} that
$\bar J_{d,g}$ has canonical singularities over a base field $k$ of ${\rm char}(k)=0$, at least if $g\geq 4$.
We then deduce some consequences for the birational geometry of the universal Jacobians $J_{d,g}$.  The outline of this section is as follows.  In \S \ref{SS:locrings}, we relate the local rings of the universal compactified Jacobian to the rings appearing earlier in this paper.  The culmination is Theorem \ref{T:loc-ring}, which essentially reduces the problem to studying finite quotients of the cographic rings $U(\Gamma)$.  In order to describe this quotient, it is convenient to compare with an associated quotient obtained from the local structure of $\overline M_g$;  this comparison is made in \S \ref{SS:loc-mor}, culminating in Theorem \ref{T:loc-mor}.   In \S \ref{SS:sing-Jac}, we give the proof of   Theorem \ref{T:can-sing}.   The argument relies on a generalization  of the Reid--Tai--Shepherd-Barron criterion  to singular toric varieties, which can be found in Appendix \ref{S:RT-toric}.   Consequences for the birational geometry of $\bar J_{d,g}$ are given in \S \ref{SS:bir-Jac}.

\subsection{The local rings of $\bar J_{d,g}$}\label{SS:locrings}

In this subsection, which is heavily based on our previous work \cite{CMKVb}, we obtain an explicit description of  the completed local rings of $\bar J_{d,g}$ in terms of the cographic toric rings studied in the previous sections.

Fix a point $(C,I)\in \bar J_{d,g}$; i.e.~$C$ is a stable curve of genus $g$, and $I$ is a rank $1$, torsion-free sheaf of degree $d$ on $C$, which is poly-stable with respect to the canonical polarization $\omega_C$.
Let $\Sigma_{(C,I)}$ (or simply $\Sigma$ when the pair $(C,I)$ we are dealing with is clear from the context) be the set of nodes of $C$ where $I$ is \emph{not} locally free.  Let $\Gamma_{(C,I)}$ (or simply $\Gamma$ when the pair
$(C,I)$ we are dealing with is clear from the context) be the graph obtained from
the dual graph of $C$ by contracting the edges corresponding to the nodes that are not in $\Sigma_{(C,I)}$.
In particular, the edges of $\Gamma_{(C,I)}$ correspond naturally to the nodes in $\Sigma_{(C,I)}$.  Note that $\Gamma_{(C,I)}$ is the dual graph of the curve obtained from $C$ by smoothing the nodes at which $I$ is locally free. For convenience, we fix an arbitrary orientation of $\Gamma_{(C,I)}$ and we denote by $s,t:E(\Gamma_{(C,I)})\to V(\Gamma_{(C,I)})$ the source and target maps, associating to any edge of $\Gamma_{(C,I)}$ the source and target with respect to the chosen orientation.

We now  review the deformation theory of the pair $(C,I)$, referring to \cite{CMKVb} for more details and  proofs.
As explained in \cite[\S 3]{CMKVb}, the \emph{deformation functor} $\Def_{(C,I)}$ of the pair $(C,I)$ fits into the following sequence
\begin{equation}\label{E:seq-def}
\Def_{(C,I)}^{\rm l.t.}\longrightarrow \Def_{(C,I)} \stackrel{F}{\longrightarrow} \prod_{e\in \Sigma}\Def_{(C_e,I_e)}=\Def_{(C,I)}^{\rm loc},
\end{equation}
where $\Def_{(C_e,I_e)}$ is the deformation functor of the pair consisting of $C_e:=\Spec \wh{\mathscr O}_{C,e}$ and the pull-back $I_e$ of $I$ to $C_e$,  $F$ is the forgetful map
mapping taking deformations of $(C,I)$ to local deformations at the set of nodes $e\in \Sigma$ where $I$ fails to be locally free, and $\Def_{(C,I)}^{\rm l.t.}$ is the subfunctor of $\Def_{(C,I)}$ parametrizing locally trivial
deformations, i.e.~deformations of $(C,I)$ that map to the trivial deformation via the forgetful map $F$.
The above three deformation functors are unobstructed and the forgetful map $F$ is  formally smooth (see \cite[\S 3]{CMKVb}). In particular, we get an exact sequence of tangent spaces
\begin{equation}\label{E:seq-tan}
0\to T \Def_{(C,I)}^{\rm l.t.}\longrightarrow T \Def_{(C,I)}\longrightarrow  T\Def_{(C,I)}^{\rm loc}\to 0.
\end{equation}
Define the following $k$-algebra
\begin{equation}\label{E:ringR}
R_{(C,I)}:=k[T^{\vee}\Def_{(C,I)}]=\bigoplus_{n\in \bbN} \Sym^n T^{\vee}\Def_{(C,I)},
\end{equation}
where $T^{\vee}\Def_{(C,I)}$ is the dual of the tangent space $T\Def_{(C,I)}$. Fixing a splitting of the exact sequence \eqref{E:seq-tan} and using the explicit description of a miniversal deformation ring
for $\Def_{(X_e,I_e)}$ obtained in \cite[Lemma 3.14]{CMKVb}, we can write $R_{(C,I)}$ in the following form
\begin{equation}\label{E:Rdec}
R_{(C,I)}=k[T^{\vee}\Def_{(C,I)}^{\rm loc}]\otimes_k k[T^{\vee}\Def^{\rm l.t.}_{(C,I)}]=\bigotimes_{e\in \Sigma} \frac{k[X_e,Y_e,T_e]}{(X_eY_e-T_e)}  \otimes_k k[T^{\vee}\Def^{\rm l.t.}_{(C,I)}]=
B(\Gamma)\otimes_k k[T^{\vee}\Def^{\rm l.t.}_{(C,I)}],
\end{equation}
where $B(\Gamma)$ is the ring defined in  Remark \ref{R:B-action}. As proved in \cite[\S 3.2]{CMKVb}, the \emph{mini-versal deformation ring} of the functor $\Def_{(C,I)}$ is given by the completion $\widehat R_{(C,I)}$ of $R_{(C,I)}$  at the maximal ideal $\m_0$ generated by $T^{\vee}\Def_{(C,I)}$.
Geometrically,   the variables $X_e$ and $Y_e$ correspond to the deformations of $I$ at the node $e\in E(\Gamma)=\Sigma$ and the variable $T_e$ corresponds to the smoothing of $C$ at $e$.
Note also,  the completion $\widehat B(\Gamma)$ of $B(\Gamma)$ at the maximal ideal generated by $T^{\vee}\Def_{(C,I)}^{\operatorname{loc}}$ was shown to be mini-versal for $\Def_{(C,I)}^{\rm loc}$; for this reason we will sometimes also write
$$
R_{(C,I)}^{\operatorname{loc}}:=B(\Gamma)=\bigotimes_{e\in \Sigma} \frac{k[X_e,Y_e,T_e]}{(X_eY_e-T_e)}.
$$

Consider now the \emph{automorphism group} $\Aut(C,I)$ of $(C,I)$,  consisting of all the pairs $(\sigma, \tau)$ such that $\sigma:C \stackrel{\cong}{\longrightarrow} C$ is an automorphism of $C$ and
$\tau: I\stackrel{\cong}{\longrightarrow} \sigma^*(I)$ is an isomorphism of sheaves on $C$. We have a natural exact sequence of groups
\begin{equation}\label{E:seq-groups}
\{1\}\to \operatorname{Aut}(I)  \stackrel{i}{\longrightarrow}  \operatorname{Aut}(C,I) \stackrel{p}{\longrightarrow}  {\rm Stab}_C(I) \longrightarrow \{1\},
\end{equation}
where ${\rm Stab}_C(I)\subseteq \operatorname{Aut}(C)$  the subgroup of $\Aut(C)$ (which is finite since $C$ is stable) consisting of all the elements $\sigma\in \Aut(C)$ such that $\sigma^*(I)\cong I$.
The group $\Aut(I)$ is an algebraic torus, which by \cite[Remark. 5.9]{CMKVb} is naturally isomorphic to
\begin{equation}\label{E:group2}
\Aut(I)= T_{\Gamma}:=\prod_{v \in V(\Gamma)} \Gm.
\end{equation}

The automorphism group $\Aut(C,I)$ acts naturally on $\Def_{(C,I)}$ (see \cite[Def. 3.4]{CMKVb}); hence it acts also on the tangent space $T\Def_{(C,I)}$ and this action clearly preserves the exact
sequence \eqref{E:seq-tan}. Therefore we get a natural linear \emph{action} of $\Aut(C,I)$ on $R_{(C,I)}$ which preserves the decomposition of $R_{(C,I)}$ given in \eqref{E:Rdec}.
It follows from \cite[\S5]{CMKVb} that the induced action of the subgroup $\Aut(I)$ is trivial on $k[T^{\vee} \Def_{(C,I)}^{\rm l.t.}]$, and  coincides with the action of $T_{\Gamma}$ on $B(\Gamma)$ given
by \eqref{E:act-coord} after the identification of Remark \ref{R:B-action}. Explicitly, an element $\lambda=(\lambda_v)_{v\in V(\Gamma)}\in T_{\Gamma}$ acts on the generators of $B(\Gamma)$ as
\begin{equation}\label{E:expl-act}
\lambda\cdot X_e=\lambda_{s(e)}\lambda_{t(e)}^{-1} X_e,  \hspace{1cm} \lambda\cdot Y_e=\lambda_{s(e)}^{-1}\lambda_{t(e)} Y_e \hspace{0.7cm} \text{ and }Ê\hspace{0.7cm} \lambda\cdot T_e=T_e.
\end{equation}

The subring  $R_{(C,I)}^{\Aut(C,I)}\subseteq R_{(C,I)}$ of invariants for the action of $\Aut(C,I)$ on $R_{(C,I)}$ can be computed in two steps: we first take the subring $R_{(C,I)}^{\Aut(I)}\subset R_{(C,I)}$
of invariants for the subgroup $\Aut(I)$; then we take the invariants for the induced action of the finite group $\Stab_C(I)$ on $R_{(C,I)}^{\Aut(I)}$. Using Theorem \ref{T:pres-quot}, the ring of invariants with respect to $\Aut(I)$ is equal to
\begin{equation}\label{E:par-inv}
R_{(C,I)}^{\Aut(I)}=U(\Gamma)\otimes_k k[T^{\vee}\Def_{(C,I)}^{\rm l.t.}],
\end{equation}
where $U(\Gamma)$ is the cographic toric ring associated to $\Gamma$. Therefore the subring of invariants with respect to $\Aut(C,I)$ is given by
\begin{equation}\label{E:tot-inv}
R_{(C,I)}^{\Aut(C,I)}=\left(R_{(C,I)}^{\Aut(I)}\right)^{\Stab_C(I)}=\left(U(\Gamma)\otimes_k k[T^{\vee}\Def_{(C,I)}^{\rm l.t.}] \right)^{\Stab_C(I)}.
\end{equation}

We show next that the completion of the invariant subring  \eqref{E:tot-inv} at the maximal  ideal $\m_0\cap R_{(C,I)}^{\Aut(C,I)}$ gives a description of the completed local ring $\wh{\OO}_{\bar J_{d,g},(C,I)}$ of the universal  compactified Jacobian $\bar J_{d,g}$ at $(C,I)$.

\begin{teo}\label{T:loc-ring}
Notation as above. Assume that $\operatorname{Stab}_C(I)$ does not contain elements of order equal to $p={\rm char}(k)$.
The completion of the invariant subring $R_{(C,I)}^{\Aut(C,I)}$ at the  maximal ideal $\m_0\cap R_{(C,I)}^{\Aut(C,I)}$ is isomorphic to the completed local ring $\wh{\OO}_{\bar J_{d,g},(C,I)}$ of
the universal compactified Jacobian $\bar J_{d,g}$ at $(C,I)$.
\end{teo}
\begin{proof}
The linear action of $\Aut(C,I)$ on $R_{(C,I)}$ described above induces a unique action on the completion $\wh{R}_{(C,I)}$ of $R_{(C,I)}$ at the maximal ideal $\m_0$.
Since $\Aut(C,I)$ is a linearly reductive group (by our assumption on $\Stab_C(I)$),  the formation of $\Aut(C,I)$-invariants commutes with completion
(see e.g.~\cite[Lemma 6.7]{CMKVb}), or in symbols

\begin{equation}\label{E:inv-comp}
\left(R_{C,I}^{\Aut(C,I)}\right)^{\widehat \ }\cong \left(\wh{R}_{(C,I)}\right)^{\Aut(C,I)}
\end{equation}
where on the right hand side the completion is taken with respect to the maximal ideal $\m_0$ of $\widehat{R}_{(C,I)}$ generated by $T^{\vee}\Def_{(C,I)}$ and on the left the completion is taken with
respect to the maximal ideal $\m_0\cap R_{(C,I)}^{\Aut(C,I)}$.

As observed before, the ring $\wh{R}_{(C,I)}$ is the mini-versal deformation ring of the functor $\Def_{(C,I)}$, which means that there is a formally smooth natural transformation of functors
\begin{equation}\label{E:nat-tran}
\Phi: \Spf \wh{R}_{(C,I)}\to \Def_{(C,I)}
\end{equation}
whose associated map on tangent spaces
\begin{equation}\label{E:map-tan}
T\Phi: T \Spf \wh{R}_{(C,I)} \longrightarrow T \Def_{(C,I)}
\end{equation}
is an isomorphism. Explicitly, the isomorphism $T\Phi$ is obtained by first identifying
the tangent space of $\Spf \wh{R}_{(C,I)}$ with the tangent space $T_{\m_0}R_{(C,I)}=(\m_0/\m_0^2)^{\vee}$ of the ring $R_{(C,I)}$ at $\m_0$ and then by identifying
 $T_{\m_0}R_{(C,I)}$ with $T\Def_{(C,I)}$ using the definition \eqref{E:ringR} of $R_{(C,I)}$.

 Observe now that our specified linear action of $\Aut(C,I)$ on $R_{(C,I)}$ is defined in such a way that
 the isomorphism $T\Phi$ becomes $\Aut(C,I)$-equivariant. Using Rim's arguments (see \cite{Rim}), the $\Aut(C,I)$-equivariance of $T\Phi$ implies that also $\Phi$ is
 $\Aut(C,I)$-equivariant;
hence the specified action of $\Aut(C,I)$ on $\wh{R}_{(C,I)}$ is the unique action that makes $\Phi$ equivariant, according to Rim's theorem (see \cite[Fact 5.4]{CMKVb}).
Therefore, we can apply \cite[Thm.~6.1(ii)]{CMKVb} in order to conclude that
\begin{equation}\label{E:loc-comp}
\wh{\OO}_{\bar J_{d,g},(C,I)}\cong \wh{R}_{(C,I)}^{\Aut(C,I)}.
\end{equation}
The proof of the theorem follows by putting together \eqref{E:inv-comp} and \eqref{E:loc-comp}.
\end{proof}

\subsection{The local structure of the morphism $\pi: \bar J_{d,g}\to \ov M_g$}\label{SS:loc-mor}

The aim of this subsection is to study the local structure of the morphism $\pi:\bar J_{d,g}\to \ov M_g$ around a point $(C,I)\in \bar J_{d,g}$, where we assume as usual that $I$ is poly-stable with respect to $\omega_C$.

First of all, there is a  natural forgetful morphism $\Pi: \Def_{(C,I)}\to \Def_C$, from the deformation functor of the pair $(C,I)$ to the deformation functor of $C$, which is equivariant with respect to the group homomorphism $\Aut(C,I)\to \Aut(C)$
and the natural actions of $\Aut(C,I)$ on $\Def_{(C,I)}$ and of $\Aut(C)$ on $\Def_C$ (see \cite[Def. 3.4]{CMKVb}). The forgetful morphism $\Pi$ fits into the following diagram
\begin{equation}\label{E:seq-def2}
\xymatrix{
\Def_{(C,I)}^{\rm l.t.}\ar[r] \ar[d]& \Def_{(C,I)} \ar[r] \ar[d]^{\Pi} & \prod_{e\in \Sigma}\Def_{(C_e,I_e)}=\Def_{(C,I)}^{\rm loc} \ar[d]\\
\Def_{C}^{\Sigma, \rm l.t.}\ar[r] & \Def_{C} \ar[r] & \prod_{e\in \Sigma}\Def_{C_e}=\Def_{C}^{\Sigma, \rm loc}
}
\end{equation}
where $\Def_{C}^{\Sigma, \rm loc}$ is the local deformation functor of $C$ at the nodes $\Sigma=\Sigma_{(C,I)}$  of $C$ where $I$ is not invertible, and $\Def_{C}^{\Sigma, \rm l.t.}$ is the subfunctor of $\Def_C$ parametrizing deformations of $C$ that are locally trivial around the nodes of $\Sigma$. Passing to the tangent spaces, we get the following diagram with exact rows
\begin{equation}\label{E:seq-tan2}
\xymatrix{
0\ar[r] & T \Def_{(C,I)}^{\rm l.t.}\ar[r] \ar[d]^{T \Pi^{\rm l.t.}}& T\Def_{(C,I)} \ar[r] \ar[d]^{T\Pi} & T \Def_{(C,I)}^{\rm loc} \ar[r] \ar[d]^{T\Pi^{\rm loc}} & 0\\
0 \ar[r] & T \Def_{C}^{\Sigma, \rm l.t.}\ar[r] & T\Def_{C} \ar[r] & T\Def_{C}^{\Sigma, \rm loc} \ar[r] & 0
}
\end{equation}
Observe that the map $T \Pi^{\rm l.t.}$ is surjective and its kernel can be naturally identified with the tangent space $T\Def_L$ of the deformation functor $\Def_L$, where $L$ is the line bundle on the partial normalization $g:C_{\Sigma}\to C$ of $C$ at the nodes of $\Sigma=\Sigma_{(C,L)}$ and $L$ is the unique line bundle on $C_{\Sigma}$ such that $I=g_*(L)$ (see \cite[Lemma 3.16]{CMKVb}).

Fixing a splitting of the second row of \eqref{E:seq-tan2}, we define the following $k$-algebra
\begin{equation}\label{E:ringR2}
R_C:=k[T^\vee \Def_C]=k[T^\vee\Def_{C}^{\Sigma, \rm loc}]\otimes_k k[T^\vee \Def_{C}^{\Sigma, \rm l.t.}]=\bigotimes_{e\in \Sigma} k[T_e]\otimes_k k[T^\vee \Def_{C}^{\Sigma, \rm l.t.}],
\end{equation}
where the variable $T_e$ corresponds to the smoothing of $C$ at $e$. Observe that  the finite group $\Aut(C)$ acts linearly on $R_C$, via its natural action on $T \Def_C$. The diagram \eqref{E:seq-tan2}, after choosing compatible splittings of the horizontal rows and of the left vertical column,  gives rise to an injective morphism of
$k$-algebras
\begin{equation}\label{E:mor-rings} 
\begin{aligned}
R_C=\bigotimes_{e\in \Sigma} k[T_e]\otimes_k k[T^\vee \Def_{C}^{\Sigma, \rm l.t.}]\hookrightarrow R_{(C,I)} &= \bigotimes_{e\in \Sigma} \frac{k[X_e,Y_e,T_e]}{(X_eY_e-T_e)}  \otimes_k k[T^{\vee}\Def^{\rm l.t.}_{(C,I)}]= \\
& =\bigotimes_{e\in \Sigma} \frac{k[X_e,Y_e,T_e]}{(X_eY_e-T_e)}  \otimes_k k[T^\vee \Def_{C}^{\Sigma, \rm l.t.}] \otimes_k k[T^\vee \Def_L].
\end{aligned}
\end{equation}
Consider now the action of $\Aut(I)$ on $R_{(C,I)}$ as in \S \ref{SS:locrings}. From \eqref{E:expl-act}, it follows that each $T_e$ is invariant under the action of $\Aut(I)$ so that the inclusion \eqref{E:mor-rings} factors through
\begin{equation}\label{E:mor-rings2} 
R_C=\bigotimes_{e\in \Sigma} k[T_e]\otimes_k k[T^\vee \Def_{C}^{\Sigma, \rm l.t.}]\hookrightarrow R_{(C,I)}^{\Aut(I)} = U(\Gamma)  \otimes_k  k[T^\vee \Def_{C}^{\Sigma, \rm l.t.}] \otimes_k k[T^\vee \Def_L].
\end{equation}
Note that the finite subgroup $\Stab_C(I)$ acts in a compatible way on both the above rings, while the bigger finite group $\Aut(C)$ acts only on the ring on the left.

\begin{teo}\label{T:loc-mor}
Notation as above. Assume that $\Aut(C)$ does not contain elements of order equal to $p={\rm char}(k)$. The inclusion of complete local rings
$$\wh{\OO}_{\bar M_{g},C}\hookrightarrow \wh{\OO}_{\bar J_{d,g},(C,I)}$$
induced by the surjective morphism $\pi:\bar J_{d,g}\to \ov M_ g$ coincide with the completion of the inclusion
\begin{equation}\label{E:incl1}
R_C^{\Aut(C)} \hookrightarrow R_{(C,I)}^{\Aut(C, I)} = \left(R_{(C,I)}^{\Aut(I)} \right)^{\Stab_C(I)}
\end{equation}
induced by \eqref{E:mor-rings2}, at their maximal ideals $\m_0\cap R_{C}^{\Aut(C)}$ and
$\m_0\cap R_{(C,I)}^{\Aut(C,I)}$, respectively.
\end{teo}
\begin{proof}
The assumption on the order of the elements of $\Aut(C)$ implies that $\Aut(C)$ and $\Aut(C,I)$ are linearly reductive groups. Since  the formation of invariants under the action of a linear reductive group commutes with completion  (see e.g.~\cite[Lemma 6.7]{CMKVb}), we get that the completion of the inclusion \eqref{E:incl1} is equal to the inclusion
\begin{equation}\label{E:incl2}
\wh{R}_C^{\Aut(C)} \hookrightarrow \wh{R}_{(C,I)}^{\Aut(C, I)},
\end{equation}
where the completions, done with respect to the maximal ideals $\m_0\cap R_C$ and $\m_0$ respectively, are acted upon naturally by $\Aut(C)$ and $\Aut(C,I)$ respectively.

From the discussion in \cite[\S 3]{CMKVb}, it follows that the inclusion $\wh R_C\hookrightarrow \wh R_{(C,I)}$ induces, by passing to the formal spectrum, a diagram
\begin{equation}\label{E:versal}
\xymatrix{
\Spf \wh{R}_{(C,I)}\ar[r]^{\Phi}\ar[d]& \Def_{(C,I)}\ar[d]^{\Pi}\\
\Spf \wh{R}_{C}\ar[r]^{\ov\Phi}& \Def_{C}
}
\end{equation}
where $\Phi$ realizes $\wh{R}_{(C,I)}$ as the mini-versal deformation ring of the functor $\Def_{(C,I)}$
(as discussed in the proof of Theorem \ref{T:loc-ring}) and $\ov\Phi$ realizes $\wh{R}_C$ as the universal deformation ring of $\Def_C$. Moreover, $\Phi$ is $\Aut(C,I)$-equivariant (as discussed in the proof of Theorem \ref{T:loc-ring}), $\ov\Phi$ is clearly $\Aut(C)$-equivariant (being an isomorphism of functors)
and the two vertical maps in \eqref{E:versal} are equivariant with respect to the group homomorphism $\Aut(C,I)\to \Aut(C)$.

Therefore, as an application of Luna's slice theorem (see \cite[\S 6]{CMKVb}), we get a commutative diagram
\begin{equation}\label{E:isorings}
\xymatrix{
\wh{R}_{(C,I)}^{\Aut(C,I)} & \wh{\OO}_{\bar J_{d,g},(C,I)} \ar[l]_{\cong}\\
\wh{R}_{C}^{\Aut(C)}\ar@{^{(}->}[u]& \wh{\OO}_{\bar M_{g},C}\ar@{^{(}->}[u]\ar[l]_{\cong}
}
\end{equation}
which concludes the proof.
\end{proof}

Consider now the graph $\Gamma=\Gamma_{(C,I)}$ obtained from the dual graph of $C$ by contracting the edges corresponding to nodes of $C$ where $I$ is locally free, as in \S\ref{SS:locrings}.
It follows from the above discussion that the inclusions $R_C\hookrightarrow R_{(C,I)}^{\Aut(I)}\hookrightarrow R_{(C,I)}$ are given, up to smooth factors, by
the following inclusions of $k$-algebras (with the notation of \S\ref{S:U-inv})
\begin{equation}\label{E:incl-alg}
\begin{aligned}
 k[T_e: e\in E(\Gamma)]  & \hookrightarrow U(\Gamma)=B(\Gamma)^{T_{\Gamma}}  \hookrightarrow B(\Gamma)= k[X_{\er}: \er \in \OE(\Gamma)] \\
 T_e & \mapsto X_{\er}\cdot X_{\el},
\end{aligned}
\end{equation}
where we used that $X_{\er}\cdot X_{\el}\in B(\Gamma)$ is invariant under the action of $T_{\Gamma}$  given in \eqref{E:act-coord}.
 Therefore, we get the following surjective morphism of varieties
\begin{equation}\label{E:mor-fg}
\Spec B(\Gamma)=\Spec k[X_{\er}: \er \in \OE(\Gamma)]\stackrel{f}{\twoheadrightarrow} X_{\Gamma}=\Spec U(\Gamma)\stackrel{g}{\twoheadrightarrow} \Spec k[T_e: e\in E(\Gamma)].
\end{equation}
The above morphisms are toric morphisms of affine toric varieties, which can be described using toric geometry as follows.
With the notation of \S\ref{S:X-without}, consider the following injective linear maps 
\begin{equation*}\label{E:injlin}
\begin{aligned}
\mathbb R\langle e \rangle_{e\in E(\Gamma)}& \stackrel{}{\longrightarrow} \bbH_1(\Gamma,\bbR)=\ker \bbD   \stackrel{}{\longrightarrow} \bbC_1(\Gamma,\bbR), \\
\sum_{e\in E(\Gamma)} a_e\cdot e & \mapsto \sum_{e\in E(\Gamma)} a_e (\er+\el)
\end{aligned}
\end{equation*}
which clearly preserve the integral lattices. By taking duals, we get the following surjective lattice-preserving linear maps 
\begin{equation}\label{E:surlin}
\bbC_1(\Gamma,\bbR)^{\vee} \stackrel{l}{\twoheadrightarrow} \bbH_1(\Gamma,\bbR)^{\vee}Ê\stackrel{h}{\twoheadrightarrow}  \mathbb R\langle e^\vee \rangle_{e\in E(\Gamma)}.
\end{equation}
The above three vectors spaces are endowed with standard scalar products that will be denoted with the same symbol $(\ ,\ )$ (see \S\ref{S:HomGra} and \S\ref{S:X-without}).
Inside the vector space $\bbH_1(\Gamma,\bbR)^{\vee}$, we have the cone $\sigma:=\sigma_{\Gamma}^{\vee}$ introduced in \S\ref{S:X-without}.
The rational polyhedral fan formed by $\sigma$ and all its faces corresponds to the toric variety $X_{\Gamma}$.
Using Lemma \ref{L:prop-cone}\eqref{L:prop-cone0}, it follows that $\sigma$ is equal to

$$\sigma=\conv \langle (\cdot, \er) \rangle_{\er\in \OE(\Gamma)}$$
where $\conv$ denotes the convex hull.
Set 
\begin{equation*}
\begin{aligned}
& \wh{\sigma}:=\conv \langle (\cdot, \er) \rangle_{\er\in \OE(\Gamma)}\subset \bbC_1(\Gamma,\bbR)^{\vee} , \\
& \w{\sigma}:=\conv \langle (\cdot, e) \rangle_{e \in E(\Gamma)}= \mathbb R_{\ge 0}\langle e^\vee \rangle_{e\in E(\Gamma)}\subset\mathbb R\langle e^\vee \rangle_{e\in E(\Gamma)}.
\end{aligned}
\end{equation*}
Clearly, the cone $\wh{\sigma}$ (resp.~$\w{\sigma}$) gives rise to the toric variety $\Spec k[X_{\er}: \er \in \OE(\Gamma)]$ (resp.~$\Spec k[T_e: e\in E(\Gamma)]$).
Moreover, the lattice-preserving linear maps \eqref{E:surlin} are such that $l(\wh{\sigma})=\sigma$ and $h(\sigma)=\w{\sigma}$; hence they induce morphisms
$ \Spec k[X_{\er}: \er \in \OE(\Gamma)]\stackrel{}{\twoheadrightarrow} X_{\Gamma} \stackrel{}{\twoheadrightarrow} \Spec k[T_e: e\in E(\Gamma)]$ which are easily seen to coincide with the morphisms
$f$ and $g$ of \eqref{E:mor-fg}.

\subsection{Singularities of $\overline{M}_{g}$}\label{S:SingMg} 
We recall the following result of Harris--Mumford and Ludwig. 

\begin{teo}[{\cite[Thm.~2]{HM}, \cite[Prop.~4.2.5, Cor.~4.2.6]{Lud}}] \label{teoHML}
Let $g\ge 4$, $C\in \overline M_g$, and $\phi\in \operatorname{Aut}(C)$.  Set $R_C$ to be a mini-versal space for $C$.   If  $\phi$  acts as a pseudo-reflection on $\Spec R_C$, or 
$\Spec R_C /\langle \phi\rangle $ does not have canonical singularities, then the following holds:

\begin{enumerate}
\item The curve $C$ has an \emph{elliptic tail} $E\subset C$, i.e.~an irreducible subcurve of arithmetic genus one that meets the complementary subcurve $E^c:=\ov{C\setminus E}$  in one point $p$, and $\phi$ is an \emph{elliptic tail automorphism}, i.e.~$\phi_{^|E^c}=\id_{E^c}$.  

\item The restriction $\phi_{|E}$ is an automorphism of $E$, fixing $p$, with order $n=2$, $3$, $4$ or $6$.   If  $n=4$, then $E$ is smooth with $j$-invariant equal to $1728$, and if  $n=3$ or $6$, then $E$ is smooth with $j$-invariant equal to $0$. 

\item If  $E$ is a singular elliptic curve, then  $\phi_{|E}$ has order $n=2$ and is given as follows: Denote by $\nu:E^{\nu}\to E$ the normalization of $E$ and identify $E^{\nu}$ with $\bbP^1$ in such a way that $\nu^{-1}=\infty=(1,0)$ and $\nu^{-1}(q)=\{(1,1),(-1,1)\}$.  Then $\phi_{|E}$ is induced by the involution of $\bbP^1$ sending $(x,y)$ into $(-x,y)$.

\end{enumerate}

Moreover, let $g\ge 4$, $C\in \overline M_g$ be a curve with  an elliptic tail $E$, and $\phi\in \operatorname{Aut}(C)$ be  an elliptic tail automorphism (with respect to $E$). 
Let  $\{t_1,\ldots,t_{3g-3}\}$ be coordinates of $T \Def_C$ such that $t_1$ corresponds to the smoothing of $C$ at the node $p$, and $t_2$ corresponds, if $E$ is smooth, to a coordinate for $T_{(E,p)}(M_{1,1})$  (corresponding to the $j$-invariant of $E$), or if $E$ is singular,  to the smoothing of $C$ at $q$.  Then the action of $\phi$ on $T \Def_C$ on the above coordinates is given by the following matrix
(depending on the choice of the primitive $n$-th root of unity $\zeta$):
\begin{equation}\label{E:M-matrix}
M(\phi)=\begin{cases}
\left(
  \begin{array}{ccc}
    \zeta^{1} &  &  \\
     & \zeta^0 &  \\
     &  & \II  \\
  \end{array}
\right)  & \text{ if } n=2,\\
\left(
  \begin{array}{ccc}
    \zeta^{1} &  &  \\
     & \zeta^2 &  \\
     &  & \II  \\
  \end{array}
\right)  \text{ or }
\left(
  \begin{array}{ccc}
    \zeta^{3} &  &  \\
     & \zeta^2 &  \\
     &  & \II  \\
  \end{array}
\right)
& \text{ if } n=4,\\
\left(
  \begin{array}{ccc}
    \zeta^{1} &  &  \\
     & \zeta^2 &  \\
     &  & \II  \\
  \end{array}
\right)  \text{ or }
\left(
  \begin{array}{ccc}
    \zeta^2 &  &  \\
     & \zeta^1 &  \\
     &  & \II  \\
  \end{array}
\right)
& \text{ if } n=3,\\
\left(
  \begin{array}{ccc}
    \zeta^5 &  &  \\
     & \zeta^4 &  \\
     &  & \II  \\
  \end{array}
\right)  \text{ or }
\left(
  \begin{array}{ccc}
    \zeta^1 &  &  \\
     & \zeta^2 &  \\
     &  & \II  \\
  \end{array}
\right)
& \text{ if } n=6,\\
\end{cases}
\end{equation}
where $\II$ is the suitable identity matrix.

\end{teo}

\subsection{Singularities of $\bar J_{d,g}$}\label{SS:sing-Jac}

The aim of this subsection is to prove that $\bar J_{d,g}$ has  canonical singularities if $g\geq 4$ and ${\rm char}(k)=0$.

\begin{teo}\label{T:can-sing}
Assume that ${\rm char}(k)=0$, and $g\geq 4$.  Then the universal compactified  Jacobian $\bar J_{d,g}$ has  canonical singularities for any $d\in \mathbb Z$.
\end{teo}
\begin{proof}
Since the property of having canonical singularities is invariant under localization and completion (see e.g.~\cite[Prop.~4-4-4]{Mats}), it is enough to show, by Theorem \ref{T:loc-ring},
that the affine variety
\begin{equation}\label{E:aff-mod}
\Spec\left[\left(R_{(C,I)}\right)^{\Aut(C, I)}\right]= \Spec\left[\left(R_{(C,I)}^{\Aut(I)}\right)^{\Stab_C(I)}\right]=\Spec \left(R_{(C,I)}^{\Aut(I)}\right) / \Stab_C(I)
\end{equation}
has canonical singularities for every $(C,I)\in \bar J_{d,g}$. 

Roughly speaking, the outline of the argument from this point is as follows.  We take the point $(C,I)\in \bar J_{d,g}$, and consider its image $C\in \overline M_g$.  Then we   break the argument into two parts: (1) $\overline M_g$ has canonical singularities near $C$, and  (2) $\overline M_g$ does not have canonical singularities near $C$.  In case (1), we use a generalization of the Reid--Tai criterion that can be applied to singular toric varieties (we review this generalization of Reid--Tai  in the appendix), and we obtain that $\Spec \left(R_{(C,I)}^{\Aut(I)}\right) / \Stab_C(I)$ (and hence $\bar J_{d,g}$) has canonical singularities at $(C,I)$.  In case (2), there is a short list due to Harris--Mumford of possible curves where $\overline M_g$ may  fail to have canonical singularities (see \S \ref{S:SingMg}).  In these cases, it will turn out that $\Spec \left(R_{(C,I)}^{\Aut(I)}\right)$ is smooth.  Thus we can apply the usual Reid--Tai criterion.  From the work of Harris--Mumford, and Ludwig (see \S \ref{S:SingMg}), one has an explicit description of the  actions needed for the analysis.  In the end, for case (2) the argument  is very similar to that in \cite{BFV}, and establishes that $\Spec \left(R_{(C,I)}^{\Aut(I)}\right) /\Stab_C(I)$ (and hence $\bar J_{d,g}$) also has  canonical singularities at $(C,I)$ in this case.    Technically, since we are able to focus on one automorphism of $(C,I)$ at a time, the argument is broken into somewhat finer pieces than just described, but this captures the main points.   

We now proceed to implement this strategy:

To begin, a standard result (see Theorem \ref{T:redcyc}) says that
$\Spec \left(R_{(C,I)}^{\Aut(I)}\right) / \Stab_C(I)$
has canonical singularities if and only 
if for every $\phi\in \Stab_C(I)\subseteq \Aut(C)$ the quotient
$$\Spec \left(R_{(C,I)}^{\Aut(I)}\right) / \langle \phi\rangle $$
has canonical singularities.  Thus we proceed by considering the quotients $\Spec \left(R_{(C,I)}^{\Aut(I)}\right) / \langle \phi\rangle$.

\subsection*{Case 1} \emph{
The automorphism $\phi\in \Stab_C(I)$ does not act as a pseudo-reflection on $\Spec R_C$ and $\Spec R_C /\langle \phi\rangle $ has canonical singularities}.
\vspace{0.1cm}

We will show that  $\Spec \left(R_{(C,I)}^{\Aut(I)}\right) / \langle \phi\rangle$ has canonical singularities.
We will apply Lemma \ref{L:map-smooth}, which is essentially a variation on the  Reid--Tai criterion tailored to this setting,  to the following morphism $\Psi$ induced by \eqref{E:mor-rings2} 
$$
\xymatrix{ 
\Spec \left(R_{(C,I)}^{\Aut(I)}\right) \ar@{=}[d] \ar[r]^{\Psi} &    \Spec R_C\times \Spec k[T^\vee \Def_L] \ar@{=}[d] \\
\Spec \left( U(\Gamma) \otimes   k[T^\vee \Def_{C}^{\Sigma, \rm l.t.}] \otimes  k[T^\vee \Def_L] \right) \ar[r] & \Spec \left( k[T_e:\: e\in \Sigma] \otimes  k[T^\vee \Def_{C}^{\Sigma, \rm l.t.}] \otimes k[T^\vee \Def_L]\right)
}
$$
and the natural action  of $\bbZ_r=\langle \phi \rangle$.  The added factor of $\Spec k[T^\vee \Def_L]$ on the right makes the computation more tractable.  
Let us check the hypothesis of Lemma \ref{L:map-smooth}. 

First of all, $\Psi$ is a toric morphism of affine toric varieties that acts as  the identity on the last two factors $\Spec k[T^\vee \Def_{C}^{\Sigma, \rm l.t.}]$ and $\Spec k[T^\vee \Def_L]$, and coincides with the map $g: X_{\Gamma}=\Spec U(\Gamma)\to \Spec k[T_e\: :\: e\in E(\Gamma)]$ of \eqref{E:mor-fg} on the first factor. As explained in \S\ref{SS:loc-mor}, the morphism $g$  is induced by the lattice-preserving linear map  $h: \bbH_1(\Gamma,\bbR)^{\vee}\rightarrow \mathbb R\langle e^\vee\rangle_{e\in E(\Gamma)}$  
of \eqref{E:surlin} which sends the cone $\sigma=\sigma_{\Gamma}^\vee$ associated to the toric variety $X_{\Gamma}$ to the cone $\wt{\sigma}$ corresponding to the toric variety $\Spec k[T_e\: e\in E(\Gamma)]$.  By Lemma \ref{L:prop-cone}\eqref{L:prop-cone2}, the extremal rays of $\sigma_{\Gamma}^\vee$ are given by $\langle (\cdot, \er)\rangle:=\bbR_{\geq 0}\cdot \er^\vee $, as $\er$ varies among the oriented edges $\OE(\Gamma)$ of $\Gamma$. 
 As explained in \S\ref{SS:loc-mor}, the linear map $h$ sends the extremal ray 
$\langle (\cdot, \er)\rangle$  of the cone $\sigma=\sigma_{\Gamma}^\vee$ into the extremal ray $\langle (\cdot, e)\rangle $ of $\wt{\sigma}$, where $e\in E(\Gamma)$ is the (unoriented) edge underlying $\er\in \OE(\Gamma)$.  
Furthermore, by the definition \eqref{E:surlin} it follows that $h$ sends the primitive element $(\cdot, \er)$ of the extremal ray  $\langle (\cdot, \er)\rangle$ (see Lemma \ref{L:prop-cone}\eqref{L:prop-cone3}) onto the primitive element $(\cdot, e)$ of the extremal ray  $\langle (\cdot, e)\rangle$. This shows that hypothesis \eqref{hypii} and \eqref{hypiii} of Lemma  \ref{L:map-smooth} are satisfied.

Consider now the action of $\bbZ_r=\langle \phi\rangle\subset \Stab_C(I)$ on the domain and codomain of $\Psi$. The action preserves the decompositions of the domain and codomain, and the toric structure on the smooth factor  $\operatorname{Spec}\left(k[T^\vee \Def_{C}^{\Sigma, \rm l.t.}] \otimes k[T^\vee \Def_L]\right)$ is chosen via an eigen basis for the action of $\phi$.  
Considering the modular interpretation of the other factors, the two actions preserve the tori inside the domain and codomain,  and moreover, as observed in \S\ref{SS:loc-mor},
the morphism $\Psi$ is $\bbZ_r$-equivariant.  In addition, the toric variety $ \Spec R_C\times \Spec k[T^\vee \Def_L] $ is smooth and $\bbZ_r$ acts on it without pseudo-reflections since $\phi$ does not act as a pseudo-reflection already 
on $\Spec R_C$ by assumption. This shows that the hypothesis \eqref{hypa} and \eqref{hypb} of Lemma \ref{L:map-smooth} are satisfied. 

Finally, the quotient $\Spec R_C/\langle \phi\rangle$ has canonical singularities by assumption. Using the Reid--Tai criterion \eqref{eqnRSBT} and the fact that $\phi$ does not act as a pseudo-reflection on $\Spec R_C$, this is equivalent to the fact that the age of $\phi$ on $\Spec R_C$ with respect to any primitive $r$-root of unity is greater than or equal to $1$. Of course, this remains true for the age of $\phi$ acting on the space $\Spec R_C\times \Spec k[T^\vee \Def_L]$, which implies that 
$\left( \Spec R_C\times \Spec k[T^\vee \Def_L]\right)/\langle \phi\rangle$ has canonical singularities by the Reid--Tai criterion.  

We can now apply Lemma \ref{L:map-smooth} in order to conclude that $\Spec \left(R_{(C,I)}^{\Aut(I)}\right)/\langle \phi\rangle$ has canonical singularities,  q.e.d.~for Case 1.

\subsection*{Case 2}
\emph{The automorphism $\phi\in \Stab_C(I)\subseteq \Aut(C)$  either   acts as a pseudo-reflection on $\Spec R_C$ or 
$\Spec R_C /\langle \phi\rangle $ does not have canonical singularities. }
\vspace{0.1cm}

 The analysis we are going to perform in this case is similar to the analysis that was performed in \cite[\S 4]{BFV}; however, there are two main differences: here we use the Pandharipande \cite{Pan} modular interpretation of $\bar J_{d,g}$
instead of the Caporaso \cite{caporaso} modular interpretation of $\bar J_{d,g}$ used in \emph{loc.~cit.};
moreover, we will not restrict ourself to the stable locus, contrary to \emph{loc.~cit}.

To begin, 
according to the results  of Harris--Mumford and Ludwig (see Theorem \ref{teoHML}), 
Case 2 can occur only if $C$ has an \emph{elliptic tail} $E\subset C$, i.e.~a connected subcurve of arithmetic genus one which meets the complementary subcurve $E^c:=\ov{C\setminus E}$  in one point $p$, and $\phi$ is an \emph{elliptic tail automorphism}, i.e.~$\phi_{^|E^c}=\id_{E^c}$.
We now consider two sub-cases:

\vskip .1 in 
\un{Case 2-I}: \emph{The sheaf $I$ is not locally free at $p$}.

\un{Case 2-II}: \emph{The sheaf $I$ is locally free at $p$}.
\vskip .1 in 

Note that in either case,  if $E$ is a rational elliptic tail with one node $q$, then $I$ could be locally free, or not, at $q$.

\vskip .1 in 
Consider now the  ring $R_{(C,I)}^{\Aut(I)}$ as  in \eqref{E:par-inv}.
As usual, denote by $\Gamma=\Gamma_{(C,I)}$ the graph obtained from the dual graph of $C$ by contracting all the edges corresponding to nodes of $C$ where $I$ is locally free. Moreover, denote by $\Gamma_E$ (resp.~$\Gamma_{E^c}$) the graph obtained from the dual graph of $E$ (resp.~of $E^c$) by contracting all the edges corresponding to the nodes of $E$ (resp.~of $E^c$) where $I$ is locally free.

In Case 2-II, the graph $\Gamma$ is obtained by joining the graphs $\Gamma_E$ and $\Gamma_{E^c}$ along a common vertex, and in Case 2-I by means of a separating edge corresponding to the node $p$. Therefore, from the explicit presentation of $U(\Gamma)\cong D(\Gamma)$ given in Theorem \ref{T:exp-pres} (see also Lemma \ref{L:sep-edge}), it follows that
\begin{equation}\label{E:splitU}
U(\Gamma)=
\begin{cases}
U(\Gamma_{E^c})\otimes_k U(\Gamma_E)\otimes_k k[T_p] & \text{ in Case 2-I,}\\
U(\Gamma_{E^c})\otimes_k U(\Gamma_E) & \text{ in Case 2-II.}\\
\end{cases}
\end{equation}
 The graph $\Gamma_E$ consists of a vertex with a loop if $E$ is a rational elliptic tail with one node $q$ and $I$ is not locally free at $q$; otherwise, $\Gamma_E$ has one vertex and no edges. Therefore, using Theorem \ref{T:pres-quot} (and say Example \ref{E:loop1}), we easily compute
\begin{equation}\label{E:UE}
U(\Gamma_E)=D(\Gamma_E)=
\begin{cases}
k[X_q,Y_q] & \text{ if } E \text{ has a node } q \text{ and } I \text{ is not locally free at } q, \\
k & \text{ otherwise.}
\end{cases}
\end{equation}
Consider now the automorphism $\phi\in \Stab_C(I)$. Clearly, $\phi$ acts on $U(\Gamma)$ by preserving the decomposition \eqref{E:splitU} and moreover, since $\phi_{E^c}=\id_{E^c}$  by assumption, $\phi$  acts trivially on $U(\Gamma_{E^c})$. Therefore, we have that
\begin{equation}\label{E:1quot}
\Spec (R_{(C,I)}^{\Aut(I)})/\langle \phi\rangle =
\begin{cases}
\Spec U(\Gamma_{E^c}) \times \Spec \left(U(\Gamma_E)\otimes_k k[T_p]\otimes_ k k[T^{\vee}\Def_{(C,I)}^{\rm l.t.}]\right)/\langle \phi\rangle & \text{ in Case 2-I,}\\
\Spec U(\Gamma_{E^c}) \times \Spec \left(U(\Gamma_E)\otimes_ k k[T^{\vee}\Def_{(C,I)}^{\rm l.t.}]\right)/\langle \phi\rangle & \text{ in Case 2-II.}\\
\end{cases}
\end{equation}
Since $\Spec U(\Gamma_{E^c})$ has canonical (and even terminal) singularities by Theorem \ref{T:sing}, it is enough to prove that $\Spec \left(U(\Gamma_E)\otimes_k k[T_p]\otimes_ k k[T^{\vee}\Def_{(C,I)}^{\rm l.t.}]\right)/\langle \phi\rangle$ and $\Spec \left(U(\Gamma_E)\otimes_ k k[T^{\vee}\Def_{(C,I)}^{\rm l.t.}]\right)/\langle \phi\rangle$ have canonical singularities.  Taking into account \eqref{E:UE}, we see that in both cases we are dealing with finite quotient singularities so that we can apply the classical Reid--Tai criterion (see Theorem \ref{T:critsmooth}) to check canonicity.

Before applying the criterion, recall from \eqref{E:mor-rings} the splitting
$$k[T^{\vee}\Def_{(C,I)}^{\rm l.t.}]\cong k[T^\vee \Def_{C}^{\Sigma, \rm l.t.}] \otimes_k k[T^\vee \Def_L],
$$
where $L$ is the unique line bundle on the partial normalization $g:C_{\Sigma}\to C$ of $C$ at the nodes
$\Sigma=\Sigma_{(C,I)}$ with the property  that $g_*(L)=I$.
We now want to choose a suitable basis of the vector space
\begin{equation}\label{eqnV}
V:=
\begin{cases}
T U(\Gamma_E)\oplus T k[T_p]\oplus T \Def_{C}^{\Sigma, \rm l.t.} \oplus T \Def_L & \text{ in Case 2-I,}\\
T U(\Gamma_E)\oplus T \Def_{C}^{\Sigma, \rm l.t.} \oplus T \Def_L  & \text{ in Case 2-II,}\\
\end{cases}
\end{equation}
and compute the matrix $R(\phi)$ of $\phi$ in terms of the chosen basis.
\vskip .1 in 

First observe that in both Case 2-I and 2-II, the upper left $2\times 2$ sub-matrix of $M(\phi)$ from  \eqref{E:M-matrix} appears as a block factor of the matrix $R(\phi)$. Indeed, in Case 2-I we can choose the coordinate $t_1$ of $T\Def_C$ corresponding to the smoothing of $C$ at the node $p$ as a coordinate of $T k[T_p]$, and in Case 2-II, we can choose $t_1$ as one of the coordinates of $T \Def_{C}^{\Sigma, \rm l.t.}$. Moreover, if $n>2$ (which implies that $E$ is smooth), then we can choose the coordinate $t_2$ of $T\Def_C$ coming from $T_{(E,p)}(M_{1,1})$, as one of the coordinates of $T \Def_{C}^{\Sigma, \rm l.t.}$.

We now  focus our attention on the action of $\phi$ on $T \Def_L$. 
Denote by $E^c_{\Sigma}$ (resp.~$E_{\Sigma}$) the normalization of $E^c$ (resp.~$E$) at the nodes belonging to $\Sigma$. The curve $C_{\Sigma}$ is the disjoint union of $E^c_{\Sigma}$ and $E_{\Sigma}$ in Case 2-I, while it is obtained by joined $E^c_{\Sigma}$ and $E_{\Sigma}$ at the separating point $p$ in Case 2-II. In any case, $L$ is completely determined by its restrictions $L_{|E^c_{\Sigma}}$ and $L_{|E_{\Sigma}}$, and moreover we have a decomposition
\begin{equation}\label{E:decL}
T \Def_L=T  \Def_{L_{|E_{\Sigma}}}\oplus T \Def_{L_{|E^c_{\Sigma}}}.
\end{equation}
Since $\phi_{|E^c}=\id_{E^c}$ by assumption, we have that $\phi$ acts trivially on $T \Def_{L_{|E^c_{\Sigma}}}$.

\vskip .1 in 

At this point, we have established what we need from the breakdown of Case 2 into Case 2-I and Case 2-II.  In short, in all of \emph{Case 2, the upper left $2\times 2$ sub-matrix of  $M(\phi)$ from  \eqref{E:M-matrix} will appear as a block factor of the matrix $R(\phi)$, and the action on $T \Def_L$ is determined by the action on 
 $T  \Def_{L_{|E_{\Sigma}}}\cong T_{L_{|E_{\Sigma}}}(\Pic(E_{\Sigma}))$.}
 
   Let us now examine the action of $\phi$ on $T  \Def_{L_{|E_{\Sigma}}}$.
  For this we consider 3 new subcases of Case 2:
\vskip .1 in 
\un{Case 2-i}: \emph{$E$ is smooth}

\un{Case 2-ii}: \emph{$E$ is a rational elliptic curve with one node $q$ and $I$ is locally free at $q$.}

\un{Case 2-iii}: \emph{$E$ is a rational elliptic curve with one node $q$ and  $I$ is not locally free at $q$. }

 \vskip .1 in 

We now proceed with a case by case analysis.

\un{Case 2-i}:
We are assuming that $E$ is smooth.  Consequently,  $E_{\Sigma}=E$ and $L_{|E_{\Sigma}}=I_E\in \Pic^{d_E}(E)$.
We can identify $E$ with $\Pic^{d_E}(E)$ sending $r\in E$ into
$\OO_{E}(r+(d_E-1)p)\in \Pic^{d_E}(E)$.
Since $\phi$ acts on $\Pic^{d_E}(E)$ via pull-back, if
the action of $\phi$ on $T_p(E)$ is given by the multiplication by a root of unity $\zeta$, then the action of $\phi$ on $T_{I_{E}}(\Pic^{d_E}(E))$
is given by the multiplication by $\zeta^{-1}$.  In other words, if the primitive $n$-th root of unity $\zeta$ is chosen for the matrix $M(\phi)$ from \eqref{E:M-matrix}, then here the action is given by the primitive $n$-th root of unity $\zeta^{-1}$.  
Therefore the matrix $N(\phi)$ of $\phi$ with respect to the decomposition \eqref{E:decL} is equal to (with respect to the same choice of the primitive $n$-th root of unity $\zeta$ as in the above matrix $M(\phi)$):
\begin{equation}\label{E:N-matrix}
N(\phi)=\begin{cases}
\left(
  \begin{array}{cc}
    \zeta^{1} &    \\
     &   \II  \\
  \end{array}
\right)  & \text{ if } n=2,\\
\left(
  \begin{array}{cc}
    \zeta^{3} &   \\
     &   \II  \\
  \end{array}
\right)  \text{ or }
\left(
  \begin{array}{cc}
    \zeta^{1} &   \\
     &   \II  \\
  \end{array}
\right)
& \text{ if } n=4,\\
\left(
  \begin{array}{cc}
    \zeta^{2} &   \\
     &   \II  \\
  \end{array}
\right)  \text{ or }
\left(
  \begin{array}{cc}
    \zeta^1 &    \\
     & \II  \\
  \end{array}
\right)
& \text{ if } n=3,\\
\left(
  \begin{array}{cc}
    \zeta^1 &   \\
     &  \II  \\
  \end{array}
\right)  \text{ or }
\left(
  \begin{array}{cc}
    \zeta^5 &   \\
     &  \II  \\
  \end{array}
\right)
& \text{ if } n=6,\\
\end{cases}
\end{equation}
where $\II$ is the suitable identity matrix.  Note that the first matrix in each row above corresponds to the first matrix in the corresponding row of \eqref{E:M-matrix}.    The matrix $R(\phi)$ describing the action of $\phi$ on the vector space $V$  \eqref{eqnV} contains the upper left $2\times 2$ sub-matrix of $M(\phi)$ from  \eqref{E:M-matrix} and the upper left $1\times 1$ sub-matrix of $N(\phi)$ from  \eqref{E:N-matrix} as block factors. An easy inspection of the matrices $M(\phi)$ and $N(\phi)$  reveals that the condition \eqref{eqnRSBT} of the Reid--Tai criterion
is satisfied, which shows that $V/\langle \phi\rangle$ has canonical singularities, as we wanted.

\vskip .1 in 
\un{Case 2-ii}: 
In this case we are assuming that  $E$ is a rational elliptic curve with one node $q$ and that $I$ is locally free at $q$.
Then also in this case $E_{\Sigma}=E$ and $L_{|E_{\Sigma}}=I_E\in \Pic^{d_E}(E)$. Moreover, we have that $\Pic^{d_E}(E)\cong \Gm$.
Explicitly, if we consider the normalization morphism $\nu:E^{\nu}\cong \bbP^1\to E$ and let $\nu^{-1}(q)=\{u,v\}$, then
any $\lambda\in \Gm$ determines a unique line bundle $L_{\lambda}\in \Pic^{d_E}(E)$ whose local sections
are the local sections $s$ of $\OO_{\bbP^1}(d_E)$ such that $s(u)=\lambda s(v)$.
Since, as observed before, $\phi_{|E}$ is induced by an involution of $E^{\nu}$ that exchanges $u$ and $v$, then clearly $\phi$ will send $L_{\lambda}$ into $L_{\lambda^{-1}}$.  This implies that the action of $\phi$ on $T_{I_{E}}(\Pic^{d_E}(E))$ is given by multiplication by $-1$, hence the matrix $N(\phi)$ is also in this case given by  \eqref{E:N-matrix} with $n=2$.

Therefore the matrix $R(\phi)$ describing the action of $\phi$ on the vector space $V$ contains the upper left $2\times 2$ sub-matrix $M(\phi)$ from \eqref{E:M-matrix} and the upper left $1\times 1$ sub-matrix of $N(\phi)$ from \eqref{E:N-matrix} as block factors, and we conclude as in the previous case  that $V/\langle \phi\rangle$ has canonical singularities, as we wanted.

\vskip .1 in 
\un{Case 2-iii}: 
In this case $E$ is a rational elliptic tail with one node $q$,  and  $I$ is not locally free at $q$. Observe that in this case $E_{\Sigma}=\bbP^1$ so that $T  \Def_{L_{|E_{\Sigma}}}=0$ and hence the action of $\phi$ on $T \Def_L$ is trivial. To proceed in this case we  consider instead the action of $\phi$ on $T U(\Gamma_E)$, which is a two-dimensional $k$-vector space since $U(\Gamma_E)=k[X_q,Y_q]$ by \eqref{E:UE}. Geometrically, the variables $X_q$ and $Y_q$ correspond to deforming the sheaf $I$ at $q$ along the two branches of $q$ (see \cite[\S 3]{CMKVb} for more details).
Since, as observed before, $\phi_{|E}$ is induced by an involution of the normalization $\nu:E^{\nu}\to E$ that exchanges the two branches above $q$, then $\phi$ acts on $U(\Gamma_E)=k[X_q,Y_q]$ by exchanging $X_q$ with $Y_q$. Therefore, we can diagonalize the action of $\phi$ on $T U(\Gamma_E)=\langle X_q^\vee, Y_q^\vee\rangle$ by choosing the basis $\{X_q^\vee-Y_q^\vee, X_q^\vee+Y_q^\vee\}$ in such a way that the matrix $P(\phi)$ describing the action of $\phi$ is equal to
\begin{equation}\label{E:P-matrix}
P(\phi)=
\left(
  \begin{array}{cc}
    -1 & 0   \\
     0 &   1  \\
  \end{array}
\right).
\end{equation}
Therefore, since the matrix $R(\phi)$ describing the action of $\phi$ on the vector space $V$ contains the upper left $2\times 2$ sub-matrix of $M(\phi)$ from \eqref{E:M-matrix} with $n=2$ and the matrix $P(\phi)$ of \eqref{E:P-matrix} as block factors, an easy inspection of the matrices $M(\phi)$ and $P(\phi)$  reveals that the condition \eqref{eqnRSBT} of the Reid--Tai criterion is satisfied also in this case, which shows that $V/\langle \phi\rangle$ has canonical singularities, as we wanted.     
\end{proof}

 Theorem \ref{T:can-sing} was  proved by Bini--Fontanari--Viviani in \cite{BFV} under
the assumption that $\gcd(d+1-g,2g-2)=1$, which is exactly the numerical condition on $d$ and $g$ that guarantees that $\bar J_{d,g}$ has finite quotient singularities. When this happens, one can prove  Theorem \ref{T:can-sing} by a direct
application of the Reid--Tai criterion  (see \cite[Thm.~4.8]{BFV}).

\begin{rem}\label{R:Q-Gor}
It follows from Theorem \ref{T:can-sing} that $\bar J_{d,g}$ is $\bbQ$-Gorenstein. Indeed, more is true: Fontanari showed in \cite{fon} that $\bar J_{d,g}$ is $\bbQ$-factorial. 
\end{rem}

We end this subsection with a description of the locus where $\bar J_{d,g}$ has  finite quotient singularities or
 is smooth.

\begin{pro}\label{P:smoothloc}
Let $(C,I)\in \bar J_{d,g}$ and assume that $\Aut(C)$ does not contain elements of order equal to $p={\rm char}(k)$.
Then
\begin{enumerate}[(i)]
\item \label{P:smoothloc1} $\bar J_{d,g}$ has finite quotient singularities at $(C,I)$ if and only if $\Gamma_{(C,I)}$ is tree-like, i.e.~it becomes a tree after removing all the loops around its vertices.
\item \label{P:smoothloc2} If $g\geq 4$ then $\bar J_{d,g}$ is smooth at $(C,I)$ if and only if $\Gamma_{(C,I)}$ is tree-like and $\Stab_C(I)=\{\Id\}$.
\end{enumerate}
\end{pro}
\begin{proof}
Part \eqref{P:smoothloc1}: using the presentation of the complete local ring of
$\bar J_{d,g}$ at $(C,I)$ given in Theorem \ref{T:loc-ring}, it is clear that $\bar J_{d,g}$ has finite quotient
singularities at $(C,I)$ if and only if $X_{\Gamma}=\Spec U(\Gamma)$  has finite quotient singularities, where
$\Gamma=\Gamma_{(C,I)}$. Proposition \ref{P:smooth} says that this is the case if and only if $\Gamma$ is tree-like, q.e.d.

Part \eqref{P:smoothloc2}: using Theorem \ref{T:loc-ring}, the smoothness of $\bar J_{d,g}$ at $(C,I)$ is equivalent to the smoothness of the quotient $\Spec \left(R_{(C,I)}^{\Aut(I)}\right)/$ $\Stab_C(I)$.
By part \eqref{P:smoothloc1}, we must have that $\Gamma=\Gamma_{(C,I)}$ is tree-like.
In this case, $\Spec \left(R_{(C,I)}^{\Aut(I)}\right)=X_{\Gamma}\times \Spec k[T^{\vee} \Def_{(C,I)}^{\rm l.t.}]$ is smooth by Proposition \ref{P:smooth}.

\un{Claim}: The finite group $\Stab_C(I)$ acts on $\Spec \left(R_{(C,I)}^{\Aut(I)}\right)$ without pseudo-reflections.

Indeed, consider the morphism $\Spec \left(R_{(C,I)}^{\Aut(I)}\right)\to \Spec R_C$ of smooth varieties. If $1\neq \phi\in \Stab_C(I)$ acts as a pseudo-reflection on $\Spec \left(R_{(C,I)}^{\Aut(I)}\right)$ then $\phi$ acts as a pseudo-reflection on $\Spec R_C$. It is well-known that this happens if and only if $C$ has an elliptic tail $E$ and $\phi$ is the elliptic tail involution, i.e.~$\phi_{|E^c}=\id_{E^c}$ and $\phi_E$ is the elliptic involution on $E$ (see Theorem \ref{teoHML}).   
This situation is a special case of the situation we dealt with in Case II of the proof of Theorem \ref{T:can-sing}, where in particular we verified that the age of $\phi$ (with respect to its action on $\Spec \left(R_{(C,I)}^{\Aut(I)}\right)$ 
and any primitive root of unity) is at least one. 
This easily implies that $\phi$ is not a pseudo-reflection because clearly any non trivial pseudo-reflection has age less than one since it has a unique eigenvalue different from one.

\vspace{0.2cm}

Using the Claim, we conclude the proof using a classical result of Prill \cite{Pri}, which says that 
for a finite group $G$ acting on a smooth variety $X$ without pseudo-reflections, the quotient $X/G$ is smooth if and only if $G$ is the trivial group.
\end{proof}

Part \eqref{P:smoothloc2} of Proposition \ref{P:smoothloc} generalizes \cite[Prop.~4.7]{BFV}, where the statement is proved under the assumption that $(C,I)$ belongs to the stable locus of $\bar J_{d,g}$, i.e.~$I$ is stable with respect to $\omega_C$.

\begin{rem}
From Proposition \ref{P:smoothloc}\eqref{P:smoothloc1}, it follows that the locus where $\bar J_{d,g}$ has finite quotient singularities is, in general, strictly bigger than:
\begin{itemize}
\item The stable locus of $\bar J_{d,g}$, which coincides with the locus of points $(C,I)$ such that
$\Aut(I)=\Gm$, or equivalently $\Gamma_{(C,I)}$ has a unique vertex.
\item The locus where the fibers of the morphism $\bar J_{d,g}\to \Mgbar$ have finite quotient singularities, which coincides with the locus of points $(C,I)$ where $I$ fails to be locally free only at separating nodes of $C$, or equivalently where $\Gamma_{(C,I)}$ is a tree (see \cite[Thm.~B]{CMKVb}).
\end{itemize}
\end{rem}

\subsection{Birational geometry of $\bar J_{d,g}$}\label{SS:bir-Jac}

The Kodaira dimension of $J_{d,g}$ was computed by Bini--Fontanari--Viviani in \cite{BFV} under the numerical
assumption that $\gcd(d+1-g,2g-2)=1$ (or $g\ge 22$; see Remark \ref{remUK}).
However,  the only place where the authors of \emph{loc.~cit.}~need the hypothesis that $\gcd(d+1-g,2g-2)=1$ is to establish that $\bar J_{d,g}$ has canonical singularities,  as they observe in the discussion following \cite[Thm.~1.4]{BFV}.
Therefore, as a corollary of \cite{BFV} and Theorem \ref{T:can-sing}, we obtain the following result  describing the Kodaira dimension of $J_{d,g}$.

\begin{cor}\label{C:dim-Kod}
Assume that ${\rm char}(k)=0$. The Kodaira dimension of the universal Jacobian $J_{d,g}$ is given by
$$\kappa(J_{d,g})=
\begin{cases}
-\infty & \text{ if } g\leq 9, \\
0 & \text{ if } g=10, \\
19 & \text{ if } g= 11, \\
3g-3 & \text{ if } g\geq 12.
\end{cases}
$$
\end{cor}

\begin{proof}  We sketch the proof for the convenience of the reader.
Verra has shown that $J_{d,g}$ is unirational for $g\leq 9$ (\cite[Thm.~1.2]{Ver}).   So let us consider the case where $g\ge 10$.
Let $\pi:\bar J_{d,g}\to \overline M_g$ be the natural forgetful map.  Using Grothendieck--Riemann--Roch, it is shown in \cite[Thm.~1.5]{BFV} that for $g\ge 4$, $K_{\bar J_{d,g}}=\pi^*(14\lambda -2\delta)$ ($=\pi^*K_{\overline M_g}+\pi^*\lambda$, agreeing with  the naive computation  over $M_g^\circ$).  As $\pi$ has connected fibers, the Iitaka dimension of  $K_{\bar J_{d,g}}$ and $14\lambda -2\delta$ are the same.
The Iitaka dimension of $14\lambda -2\delta$ is by now well known: $\kappa(14\lambda -2\delta)=0$  if $g=10$,
$\kappa(14\lambda -2\delta)=19$ if $ g= 11$ and $\kappa(14\lambda -2\delta)=
3g-3$ if $  g\geq 12$.    (Recall that for $g\ge 13$, work of Eisenbud, Harris and Mumford \cite{HM, EH} shows that the slope of $\overline M_g$ satisfies $s(\overline M_g)<7$, and recent work of Cotterill \cite{cotterill12} shows the same holds for $g=12$.  For $g=10, 11$, work of Tan \cite{tan98} and Farkas--Popa \cite{FP05} shows that $s(\overline M_g)$=7; in these cases $\kappa(14\lambda -2\delta)$ is worked out directly in 
\cite[\S 6]{BFV}.)
Finally, since in Theorem \ref{T:can-sing} we have shown that
$\bar J_{d,g}$   has canonical singularities, we can conclude that $\kappa(\bar J_{d,g})=\kappa(K_{\bar J_{d,g}})$, completing the proof.
\end{proof}

\begin{rem}\label{remUK}
From general results of Ueno   \cite[Thm.~6.12]{ueno75}  and Kawamata  \cite[Cor.~1.2]{kawamata85},  using the fact that the Kodaira dimension of an abelian variety is zero, one obtains the estimate on the Kodaira dimension:
$\kappa(\overline M_g)\le \kappa(\bar J_{d,g})\le \dim \overline M_g$.  By virtue of the results of Harris--Mumford, Eisenbud--Harris and Farkas, that $\overline M_g$ is of general type  for $g=22$, $g\ge 24$, one obtains immediately that $\kappa(\bar J_{d,g})=\kappa(\overline M_g)=3g-3$ for $g$ in this range.
\end{rem}

\begin{rem}
Since the generic fiber of $\pi:\bar J_{d,g}\to \overline M_g$ has trivial canonical bundle, it is interesting to compare the Kodaira dimensions of the two spaces.  For the convenience of the reader, in the table below we compile the current state of the art on the Kodaira dimension of $\overline M_g$ (we refer the reader to Farkas \cite{far2} for references), and compare it with the Kodaira dimension of $\bar J_{d,g}$.  
\end{rem}

\begin{equation}\label{eqnKDC}
\begin{array}{c|c|c|c|c|c|c|c|c|c|c}
&g\le 7&8&9&10&11&12\le g\le 16&17\le g\le 21&22&23&24\le g\\ \hline
\kappa(\overline M_g)&-\infty&-\infty&-\infty&-\infty&-\infty&-\infty&\text{unknown}&3g-3&\geq 2 &3g-3\\
 \kappa(\bar J_{d,g})&-\infty&-\infty&-\infty&0&19&3g-3&3g-3&3g-3&3g-3&3g-3\\\hline \hline
 \kappa({\mathcal S}_g^-)&-\infty&-\infty&-\infty&-\infty &-\infty &3g-3&3g-3&3g-3&3g-3&3g-3\\
  \kappa({\mathcal S}_g^+)&-\infty&0&3g-3&3g-3 &3g-3 &3g-3&3g-3&3g-3&3g-3&3g-3\\
\end{array}
\end{equation}

\begin{rem}\label{remFVthch}
In recent work Farkas--Verra \cite{Far-theta, FVNik, farkasThCh, farkasverraOddThCh}  have computed the Kodaira dimension of the moduli of spin curves; i.e., the moduli of pairs consisting of a curve together with a theta characteristic.  For each $g\ge 2$, the space has two components, $\mathcal S_g^+$ and $\mathcal S_g^-$ corresponding to the even and odd theta characteristics.  Since these sit inside $ J_{g-1,g}$, \'etale over $M_g$, we find it interesting to compare the Kodaira dimensions of these spaces   \eqref{eqnKDC}.   It turns out, for instance, that both $\bar J_{d,g}$ and $\mathcal S_g^-$ attain ``maximal'' Kodaira dimension at $g=12$.  
\end{rem}

In \cite[Prop.~6.3, Prop.~6.5]{BFV} the Iitaka fibration of the canonical class  $K_{\bar J_{d,g}}$ is established  for $g\ge 10$.  This provides the Iitaka fibration for $ J_{d,g}$ under the additional hypothesis that $\bar J_{d,g}$ has canonical singularities.      
Consequently,  \cite{BFV} have determined  the Iitaka fibration for $J_{d,g}$ assuming that  $\gcd(d+1-g,2g-2)=1$ (and also for $g\ge 22$ using a different argument; see \cite[Prop.~3.2]{BFV}).  As a consequence of Theorem \ref{T:can-sing}, we obtain the following result, generalizing those of \cite{BFV}.

\begin{cor}\label{C:ItFi} For $g\ge 10$, the Iitaka fibration of $ J_{d,g}$ is given as follows:
\begin{enumerate}
\item For $g\ge 12$, the Iitaka fibration is the forgetful  morphism $\pi:\bar J_{d,g}\to \overline M_g$.

\item For $g=11$, the Iitaka fibration is the rational map $\bar J_{d,11}\dashrightarrow \mathcal F_{11}$, where $\mathcal F_{g}$ is the moduli of K3 surfaces with polarization of degree $2g-2$, and the rational map takes a  general pair $(C,L)$ to the pair $(S,\mathcal O_S(C))$, where $S$ is the unique K3 containing $C$ (see \cite{mukai96}).
  
\item For $g=10$, the Iitaka fibration is the structure  morphism $\bar J_{d,10}\to \operatorname{Spec} k$.  
\end{enumerate}
\end{cor}

\begin{proof}  We sketch the proof for the convenience of the reader.  
For $g\ge 12$, this follows from Theorem \ref{T:can-sing} and \cite[Thm.~6.11]{ueno75}.  Indeed, let $\widetilde M_g$ be a resolution of singularities  of $\overline M_g$, and let $\tilde J_{d,g}$ be a resolution of singularities of the fiber product $\bar J_{d,g}\times_{\overline M_g} \widetilde M_g$.  Then   the morphism $\tilde \pi: \tilde  J_{d,g}\to \widetilde M_g$  of smooth projective varieties is an
algebraic fiber space such that $\dim  \widetilde M_g = \kappa (\tilde J_{d,g})$ and the generic fiber
$\tilde \pi^{-1}(C) = J^dC$ is smooth and irreducible of Kodaira dimension zero.    The same argument works for $g=10$, using a desingularization $\tilde J_{d,10}$ of $\bar J_{d,10}$.    For $g=11$, we refer the reader to  \cite[Prop.~6.5]{BFV}, where it is shown that the rational map $\bar J_{d,11}\dashrightarrow \mathcal F_{11}$ is the Iitaka fibration for $K_{\bar J_{d,11}}$.  Since $\bar J_{d,g}$ has canonical singularities by Theorem \ref{T:can-sing}, it follows that this rational map is the Iitaka fibration for $J_{d,11}$. 
\end{proof}

In the last section of \cite{BFV}, the authors investigate the birational maps among the different universal Jacobians
$J_{d,g}$, as $d$ varies. Using Theorem \ref{T:can-sing}, we can relax their hypothesis (see the discussion at the end
of \cite[\S7]{BFV}).

\begin{cor}\label{C:bir-Jac}
Assume that ${\rm char}(k)=0$ and that $g\geq 12$. If $\eta: J_{d,g}\dashrightarrow J_{d',g}$ is a birational map then
$d'=\pm d+n(2g-2)$ and $\eta$ is given by the map sending $(C,L)\in J_{d,g}$ into
$(C,L^{\pm 1}\otimes \omega_C^n)\in J_{d',g}$. In particular:
\begin{enumerate}[(i)]
\item $J_{d,g}$ is birational to $J_{d',g}$ if and only if $d'\equiv \pm d \mod 2g-2$.
\item The group $\Bir(J_{d,g})$ of birational automorphisms of $J_{d,g}$ is given by
$$\Bir(J_{d,g})=
\begin{cases}
\bbz/ 2\bbz & \text{ if } d=n(g-1)  \text{ for some } n\in \bbz,\\
\{\Id \} & \text{ otherwise.}
\end{cases}
$$
Moreover, if $d=n(g-1)$ for some $n\in \bbz$ then the generator of $\Bir(J_{d,g})$ is the birational automorphism sending $(C,L)$ into $(C, L^{-1}\otimes \omega_C^n)$.
\end{enumerate}
\end{cor}

\begin{proof}
We sketch the proof for the convenience of the reader.  
As established in Corollary  \ref{C:ItFi},  for $g\ge 12$, the morphism $\pi:\bar J_{d,g}\to \overline M_g$ is the Iitaka fibration of $J_{d,g}$.      
It follows that any birational automorphism $\eta:\bar J_{d,g}\dashrightarrow \bar J_{d',g}$ induces a commutative diagram of rational maps (see e.g.~\cite[Ch.~II, Thm.~6.11]{ueno75})
$$
\xymatrix{
\bar J_{d,g} \ar@{-->}[r]^\eta \ar@{->}[d]^\pi& \bar J_{d',g} \ar@{->}[d]^{\pi}\\
\overline M_g \ar@{-->}[r]^\xi & \overline M_g\\
}
$$
The rational map $\xi$ is the identity.  Indeed, indeed if $C\in M_g$ is very general, and $C'=\xi(C)$, then there is an induced birational map $J^dC\dashrightarrow J^{d'}C'$.  As this is a birational map of abelian varieties, it is an isomorphism, and one concludes that $C\cong C'$ using the Torelli theorem, and the fact that for a very general curve, the Neron--Severi group of the Jacobian is isomorphic to $ \mathbb Z$ (see \cite[Lem.~7.4]{BFV} for more details). 
Having established that  $\xi$ is the identity, the corollary follows from \cite[Prop.~3.2.2]{caporasoNotes}.  Again, we sketch the proof for the convenience of the reader.  
 Let $U\subseteq M_g^\circ$ be an open set over which $\eta$ is defined.  For each $C\in U$, there is an isomorphism $\eta|_C:J^dC\to J^{d'}C$.    Since an isomorphism of abelian varieties is given by a translation, followed by a group automorphism, and $C$ is automorphism free, then  $\eta|_C(L)=(L\otimes L_C)^{\pm 1}$ for some $L_C\in J^{\pm(d'-d)}C$, depending only on $C$  (see \cite[Lem.~3.2.3, Prop.~3.2.2, p.16]{caporasoNotes} for more details).      The assignment $C\mapsto L_C$ determines a rational section of $J_{\pm (d'-d),g}\to M_g$.   The Franchetta Conjecture (proven by \cite{mestrano87}) asserts that the only such sections are given by pluricanonical bundles.    
\end{proof}

\begin{rem}
It is likely that Corollary \ref{C:bir-Jac} fails for small values of $g$, where it is natural to expect that
$J_{d,g}$ is rational for all values of $d\in \bbz$.
\end{rem}

\appendix

\section{Finite quotients of toric singularities}\label{S:RT-toric}

The aim of this appendix is to study when a finite quotient of a toric singularity is Gorenstein, resp.~terminal, resp.~canonical.  We will work over an algebraically closed field $k$ of characteristic $0$.  The main focus is to generalize the Reid--Tai--Shepherd-Barron criterion for quotients of smooth varieties by finite groups.  We expect these type of results are well-known to the experts, but we were unable to find a reference for the specific results we use, and so we include statements and proofs here.

\subsection{Finite quotient of smooth varieties}\label{SS:smooth}

Let us start by recalling the case of finite quotients of smooth varieties which is well-known and attributed to Khinich, Watanabe, Tai, Reid--Shepherd-Barron and Reid (see e.g.~\cite[Thm.~2.3]{MS} and the references therein).

\begin{teo}\label{T:critsmooth}
Let $G\subseteq \GL_n(k)$  be a finite subgroup and assume that $G$ does not contain pseudo-reflections.  Set $X=\bbA_k^n/G$.     For each $g\in G$ of order $r\ne 1$ and each primitive $r$-th root of unity $\zeta$, write the eigenvalues of $g$ as $\zeta^{a_1},\ldots,\zeta^{a_n}$ with $0\le a_i<r$ and define  the {\rm age}  of $g$ with respect to $\zeta$ as
$$ \age(g,\zeta):= \frac{1}{r}\sum_{i=1}^na_i.$$

\begin{enumerate}
\item \emph{(Khinich and Watanabe)} $X$ is Gorenstein if and only if $G\subseteq \SL_n(k)$; i.e.,
\begin{equation}\label{eqnKW}
 \age(g,\zeta)\in \bbZ
\end{equation}
for each  $1\neq g\in G$ and each (or, equivalently, some) primitive $r$-th root of unity $\zeta$.

\item \emph{(Reid--Shepherd-Barron \cite{Reid} and Tai \cite{tai})}  $X$ is canonical if and only if, in the notation above,
\begin{equation}\label{eqnRSBT}
 \age(g,\zeta) \ge 1
\end{equation}
for each  $1\neq g\in G$  and each primitive $r$-th root of unity $\zeta$.

\item \emph{(Reid \cite{Reid})}  $X$ is terminal if and only if, in the notation above,
\begin{equation}\label{eqnR}
 \age(g,\zeta) >1
\end{equation}
for each  $1\neq g\in G$ and each primitive $r$-th root of unity $\zeta$.
\end{enumerate}

\end{teo}

\begin{rem}\label{R:psrefle}
Recall that an element $1\neq g\in \GL_n(k)$ is a pseudo-reflection  if its fixed locus ${\rm Fix}(g):=\{x\in \bbA_k^n :  g\cdot x=x\}$  is a divisor inside $\bbA_k^n$. Equivalently, $1\neq g\in \GL_n(k)$ is a pseudo-reflection if and only if
$1$ is an eigenvalue of $g$ with multiplicity equal to $n-1$. In particular, if $1\neq g\in \GL_n(k)$ is a pseudo-reflection,  then $g\not\in \SL_n(k)$.
Note that:
\noindent \begin{enumerate}[(i)]
\item In the above theorem,  if one removes the hypothesis that $G$ has no pseudo-reflections, the conditions \eqref{eqnRSBT} and \eqref{eqnR} still imply canonical and terminal singularities, respectively.
\item If $G\subset \GL_n(k)$ is a finite group, denote by $G_{\rm ps}$ be the normal subgroup of $G$ generated by the pseudo-reflections in $G$. Then  $\bbA_k^n/G_{\rm ps}$ is smooth, i.e.
$\bbA_k^n/G_{\rm ps}\cong \bbA_k^m$ for some $m\leq n$, the quotient group $G/G_{\rm ps}$ acts linearly on $\bbA_k^m$ without pseudo-reflections and $\bbA_k^n/G\cong \bbA_k^m/(G/G_{\rm ps})$ (see \cite[\S 3.18]{Kol}).
Therefore, we can always reduce to the case of finite groups acting without pseudo-reflections.
\end{enumerate}
\end{rem}

\subsection{Notation and background results on toric varieties}\label{SS:not-tor}
We now recall some notation and background results on toric varieties, following \cite{CLS}.

Fix a lattice $N$, i.e.~a free $\mathbb Z$-module of finite rank, and let $M=N^{\vee}$ be its dual lattice.
Given a (convex, rational polyhedral) cone
$$
\sigma\subseteq N\otimes_{\mathbb Z}\mathbb R:=N_{\bbR}
$$
consider its dual cone (which is still convex, rational polyhedral)
$$
\sigma^\vee=\{\lambda \in M_{\mathbb R}:\langle \lambda, n\rangle \ge 0,\  \forall n\in \sigma\}\subset M\otimes_{\mathbb Z}\mathbb R=:M_{\mathbb R}.
$$
The affine toric variety for the torus $\bbT:=\Spec k[M]=\Gm\otimes_{\bbZ} N$ associated to $\sigma\subset N_{\bbR}$ is given  by
$$
U_\sigma=U_{\sigma,N}:=\operatorname{Spec}k[\sigma^\vee \cap M]
$$
where $k[\sigma^\vee\cap M]$ is the affine semigroup $k$-algebra associated to the normal affine semigroup $\sigma^\vee \cap M$ (by Gordon's Lemma, see \cite[Prop.~1.2.17]{CLS}).
Note that the affine toric variety $U_{\sigma, N}$ depends both on the cone $\sigma\subset N_{\bbR}$ and on the lattice $N\subset N_{\bbR}$.

In the sequel, we will use the following notation:

\begin{itemize}
\item $\sigma(1)$ is the set of one dimensional faces of $\sigma$, i.e.~the extremal rays of the cone.

\item Given $\rho\in \sigma(1)$, we set $u_\rho=u_{\rho,N}$ to be the primitive element of $\rho\cap N$.  That is $u_\rho\in \rho\cap N$, and if $u\in \rho\cap N$, then $u=nu_\rho$ for some $n\in \mathbb N$.

\item $\Pi_\sigma=\Pi_{\sigma, N}$ denotes the polytope $\Pi_\sigma=\operatorname{Conv}(0,u_{\rho,N})_{\rho\in \sigma(1)}$; i.e., the convex hull of $0$ and the primitive elements of the extremal rays of $\sigma$, with respect to the lattice $N$.

\end{itemize}
Note that the primitive elements associated to the rays of $\sigma$ depend on the lattice $N$ we are considering; therefore, also the polytope $\Pi_{\sigma,N}$ depends upon the lattice $N$. 

In what follows, we will be using the following basic results on toric singularities.

\begin{pro}[Gorenstein Condition  {\cite[Prop.~8.2.12, Prop.~11.4.11]{CLS}}]  \label{proGC}
In the notation above, the affine toric variety $U_\sigma$ is  Gorenstein if and only if  there exists $m_\sigma \in  M$ such that
$$
\langle m_\sigma ,u_\rho \rangle =1 \text{ for all } \rho \in \sigma (1).
$$
In this case, $U_\sigma$ has canonical singularities.
\end{pro}

\begin{pro}[$\mathbb Q$-Gorenstein  Condition {\cite[ Prop.~11.4.12]{CLS}}] \label{proQGC}
In the notation above, the following conditions are equivalent:
\begin{enumerate}[(i)]
\item $U_\sigma$ is $\mathbb Q$-Gorenstein.

\item There exists $m_\sigma \in  M_{\mathbb Q}$ such that
$
\langle m_\sigma ,u_\rho \rangle =1 \text{ for all } \rho \in \sigma (1)$.

\item The polytope  $\Pi_\sigma$ has a unique facet not containing the origin.

\end{enumerate}
\end{pro}

Note that the property of $U_{\sigma,N}=U_{\sigma}$ being $\bbQ$-Gorenstein depends both on the cone $\sigma$ and on the lattice $N$ (see Example \ref{E:nonQGor}). This is not the case for the stronger property of $U_{\sigma, N}=U_{\sigma}$ being $\bbQ$-factorial, which is equivalent to the cone
$\sigma$ being simplicial (see \cite[Thm.~11.4.8]{CLS}), and hence depends only on the cone $\sigma$ and not on the lattice $N$.

\begin{pro}[Canonical/Terminal Condition {\cite[Prop.~11.4.12]{CLS}}] \label{proCC}
In the notation above, assume that $U_{\sigma}$ is $\bbQ$-Gorenstein. Then $U_{\sigma}$  has canonical (resp.~terminal) singularities if and only if the only non-zero lattice points in the polytope  $\Pi_\sigma$ lie on the unique facet of
$\Pi_{\sigma}$ not containing the origin  (resp.~the only lattice points of $\Pi_\sigma$ are its vertices).
\end{pro}

\subsection{The case of cyclic groups}\label{SS:cyclic}

In this subsection, we will consider the special case of a cyclic group $\bbZ_r:=\bbZ/r\bbZ$  acting on an affine toric variety $U_{\sigma}$, preserving  the torus $\bbT=\Spec k[M]$.

After fixing a primitive $r$-th root of unity $\zeta\in k$, the action of $\mathbb Z_r$ on the coordinate ring $k[M]$ of $\bbT$ is given by a linear form $\lambda:M\to \mathbb Z$, well defined up to adding an $r$ multiple of a linear form; in other words, the action is uniquely determined by an element $[\lambda]\in N/rN=\operatorname{Hom}_{\mathbb Z}(M,\mathbb Z/r\mathbb Z)$. Explicitly, if  we choose a primitive $r$-th root of unity $\zeta\in k$, we can identify the group $\bbZ_r$ with the subgroup of $k^*$ generated by $\zeta$ and the action on $k[M]$ is given by
\begin{equation}\label{E:action1}
\zeta \cdot x^m=\zeta^{{\lambda(m)}}x^m.
\end{equation}
Moreover,  if we fix an isomorphism $M\cong \bbZ^n$ so that $k[M]=k[x_1^{\pm1},\ldots,x_n^{\pm 1}]$, then the action of $\bbZ_r$ on $k[M]$ is given by
\begin{equation}\label{E:action2}
\zeta\cdot x_i=\zeta^{a_i}x_i \hspace{0.5cm} \text{ for some } 0\le a_i<r \hspace{0.3cm} (i=1,\ldots,n).
\end{equation}

\begin{pro}\label{P:crit-cyc}
 Let $N=\mathbb Z^n=\mathbb Z\langle e_1,\ldots, e_n\rangle$, and let $\sigma\subseteq N_{\mathbb R}$ be a (convex, rational polyhedral) cone.
Let $\zeta$ be a primitive $r$-th root of unity and suppose that $\mathbb Z_r=\langle \zeta \rangle$ acts on $U_{\sigma,N}$ preserving the torus $\mathbb T =\operatorname{Spec}k[M]$ and that
the action on the ring $k[M]=k[x_1^{\pm 1},\ldots,x_n^{\pm 1}]$
is  given by
$$
\zeta \cdot x^m=\zeta^{{\lambda(m)}}x^m \hspace{0.5cm} \text{ for some }  [\lambda]\in N/rN,
$$
or more explicitly by
\begin{equation*}
\zeta\cdot x_i=\zeta^{a_i}x_i \hspace{0.5cm} \text{ for some } 0\le a_i<r \hspace{0.3cm} (i=1,\ldots,n).
\end{equation*}
Then $U_{\sigma, N}/\mathbb Z_r$ is isomorphic to the affine toric variety $U_{\sigma, N'}$ where $N'$ is the super-lattice of $N$  given by
$$
N\subseteq N'=N+\mathbb Z\left \langle \frac{1}{r}\lambda\right \rangle =\mathbb Z \left\langle e_1,\ldots,e_n, \sum_{i=1}^n\frac{a_ie_i}{r} \right\rangle\subset N_{\bbQ}.
$$
In particular, $U_{\sigma, N}/\mathbb Z_r$ is
\begin{enumerate}[(i)]
\item \label{P:crit-cyc1} $\bbQ$-Gorenstein if and only  if $\Pi_{\sigma,N'}$ has a unique facet not containing the origin;
\item \label{P:crit-cyc2} canonical if and only if $\Pi_{\sigma,N'}$ has a unique facet not containing the origin and  the only non-zero lattice  points in $\Pi_{\sigma,N'}$ lie in this facet;
\item \label{P:crit-cyc3} terminal if and only if $\Pi_{\sigma,N'}$ has a unique facet not containing the origin and the only lattice points of $\Pi_{\sigma,N'}$ are its vertices.
\end{enumerate}
\end{pro}

\begin{proof}
Let $M'\subseteq M$ be the sub-lattice of invariants; i.e., $k[M']=k[x_1^{\pm1},\ldots,x_n^{\pm 1}]^{\mathbb Z_r}$.   Clearly, the quotient $U_{\sigma,N}/\bbZ_r$ is the affine toric variety equal to $\Spec k[\sigma\cap M']$.
Therefore, in order to prove the first statement, we need to prove that after setting $N'=N+\mathbb Z\left \langle \frac{1}{r}\lambda\right \rangle$, we have $(N')^\vee=M'\subseteq M$. 
Since $N\subseteq N'$, with torsion quotient, we have $M=N^\vee \supseteq (N')^\vee$.  
Now pick an element $m=\sum_{i=1}^nm_ie_i^\vee\in M$ (with $m_i\in \bbZ$). Since
$N'$ is obtained from $N$ by adding the element $\frac{1}{r}\lambda=\sum_{i=1}^n\frac{a_ie_i}{r}\in N_{\bbQ}$, we have that
$$
m\in (N')^\vee\iff \sum_{i=1}^n\frac{a_im_i}{r}\in \mathbb Z \iff \sum_{i=1}^n a_im_i\equiv 0 \pmod r \iff x^m:=\prod_{i=1}^n x_i^{m_i}\in k[x_1^{\pm1},\ldots,x_n^{\pm 1}]^{\mathbb Z_r}$$
$$
\iff m\in M'.
$$
The assertions (i)-(iii) now follow  from this using Propositions \ref{proQGC} and \ref{proCC}.
\end{proof}

Using the above proposition, we can prove the following criterion that plays a  crucial role in the proof of Theorem \ref{T:can-sing}.

\begin{lem}\label{L:map-smooth} 
For $i=1, 2$, 
let $N_i$ be a lattice and let $\sigma_i\subset (N_i)_{\bbR}$ be a (convex, rational polyhedral) cone. Let $\phi:U_{\sigma_1,N_1}\to U_{\sigma_2, N_2}$ be a toric morphism induced by a homomorphism $\ov\phi:N_1\to N_2$ of lattices such that
\begin{enumerate}[(i)]
\item \label{hypii} $\rho\in \sigma_1(1)\Rightarrow \ov\phi_{\bbR}(\rho)\in \sigma_2(1)$;
\item \label{hypiii} For every $\rho\in \sigma_1(1)$, we have that $\ov\phi(u_{\rho,N_1})=u_{\ov\phi_{\bbR}(\rho),N_2}$.
\end{enumerate}
Suppose now that the cyclic group $\bbZ_r$ acts on the $U_{\sigma_i,N_i}$, preserving the torus $\bbT_i=\Gm\otimes_{\bbZ} N_i$ for $i=1,2$ and assume that 
\begin{enumerate}[(a)]
\item \label{hypa} $\phi: U_{\sigma_1,N_1}\to U_{\sigma_2,N_2}$ is $\bbZ_r$-equivariant;
\item \label{hypb} $U_{\sigma_2,N_2}$ is smooth and $\bbZ_r$ acts on $U_{\sigma_2,N_2}$ without pseudo-reflections;
\end{enumerate} 
Then $U_{\sigma_1,N_1}/\bbZ_r$ is $\bbQ$-Gorenstein. Moreover, if $U_{\sigma_2,N_2}/\bbZ_r$ has canonical    singularities, then $U_{\sigma_1,N_1}$ has canonical    singularities.
\end{lem}
\begin{proof}
Following the above notation, fix a primitive $r$-th root of unity $\zeta$, and suppose that the action of $\bbZ_r$ on 
$U_{\sigma_i,N_i}$ is determined by the element $[\lambda_i]\in N_i/rN_i$. Since $\phi$ is $\bbZ_r$-equivariant by 
\eqref{hypa}, we must have that $\ov\phi([\lambda_1])=[\lambda_2]$ so that the homomorphism 
$\ov\phi:N_1\to N_2$ extends to a homomorphism (which we will still denote by $\ov\phi$)
$$\ov\phi: N_1':=N_1+\mathbb Z\left \langle \frac{1}{r}\lambda_1 \right\rangle \longrightarrow N_2':=N_2+\mathbb Z\left\langle \frac{1}{r}\lambda_2\right \rangle.$$
By Proposition \ref{P:crit-cyc}, the toric morphism $\wt{\phi}:U_{\sigma_1,N_1'}\to U_{\sigma_2,N_2'}$ induced by $\ov\phi$ coincides with the quotient map $U_{\sigma_1,N_1}/\bbZ_r\to U_{\sigma_2,N_2}/\bbZ_r$ induced by $\phi$.

Fix now an extremal ray $\rho$ of $\sigma_1$ and look at $\ov\phi_{\bbR}(\rho)$, which is an extremal ray of $\sigma_2$ by  \eqref{hypii}. 
Since $N_1\subseteq N_1'$, the two primitive elements along the ray $\rho$ with respect to the above lattices are related by $u_{\rho,N_1}=c\cdot u_{\rho,N_1'}$ for some $c\in \bbZ_{>0}$.
On the other hand, it follows from  \eqref{hypb} that $u_{\ov\phi_{\bbR}(\rho),N_2}=u_{\ov\phi_{\bbR}(\rho),N_2'}$. Moreover, it follows from \eqref{hypiii}
that $\ov\phi(u_{\rho,N_1})=u_{\ov\phi_{\bbR}(\rho),N_2}$.  Finally, we will have that $\ov\phi(u_{\rho,N_1'})=l\cdot u_{\ov\phi_{\bbR}(\rho),N_2'}$ for some $l\in \bbZ_{>0}$. 
Putting everything together we find that 
$$u_{\ov\phi_{\bbR}(\rho),N_2'}=u_{\ov\phi_{\bbR}(\rho),N_2}=\ov\phi(u_{\rho,N_1})=c\cdot \ov\phi(u_{\rho,N_1'})=c\cdot l\cdot u_{\ov\phi_{\bbR}(\rho),N_2'}
$$
from which we deduce that $c=l=1$, i.e.~that 
\begin{equation}\label{eqnTorEq}
u_{\rho,N_1}= u_{\rho,N_1'} \: \text{ and } \: \ov\phi(u_{\rho,N_1'})=u_{\ov\phi_{\bbR}(\rho),N_2'}.
\end{equation}
Observe now that, since $U_{\sigma_2, N_2}$ is smooth by \eqref{hypb}, the quotient $U_{\sigma_2,N_2}/\bbZ_ r=U_{\sigma_2, N_2'}$ is $\bbQ$-factorial, hence in particular $\bbQ$-Gorenstein. By Proposition \ref{proQGC}, there exists $m_2\in (M_2')_{\bbQ}=(M_2)_{\bbQ}=(N_2^\vee)_{\bbQ}$ such that $\langle m_2, u_{\tau, N_2'}\rangle=1$ for every extremal ray $\tau$ of $\sigma_2$. Using \eqref{eqnTorEq}, the element $m_1=(\ov\phi_{\bbR})^\vee(m_2)\in (M_1')_{\bbQ}=(M_1)_{\bbQ}=(N_1^\vee)_{\bbQ}$ satisfies (for every extremal ray $\rho$ of $\sigma_1$)
$$\langle m_1, u_{\rho, N_1'}\rangle= \langle (\ov\phi_{\bbR})^\vee(m_2), u_{\rho, N_1'}\rangle= 
 \langle m_2, \ov\phi(u_{\rho, N_1'})\rangle=\langle m_2, u_{\ov\phi_{\bbR}(\rho),N_2'} \rangle=1,
$$
which shows that $U_{\sigma_1,N_1'}=U_{\sigma_1,N_1}/\bbZ_r$ is $\bbQ$-Gorenstein.

Take now a point $0\neq x\in N_1'$ which belongs to $\Pi_{\sigma_1,N_1'}$, i.e.
$$x=\sum_{\rho\in \sigma_1(1)} \alpha_{\rho}\cdot  u_{\rho, N_1'} \: \text{ with } \alpha_{\rho}\geq 0 \: \text{ and }\: 0< \sum_{\rho\in \sigma_1(1)}\alpha_{\rho}\leq 1.$$
Using \eqref{eqnTorEq}, we get that 
$$\ov\phi(x)=\sum_{\rho\in \sigma_1(1)} \alpha_{\rho}\cdot  u_{\ov\phi_{\bbR}(\rho), N_2'} 
\Rightarrow 0\neq \ov\phi(x)\in \Pi_{\sigma_2,N_2'}.$$

If $U_{\sigma_2,N_2'}=U_{\sigma_2,N_2}/\bbZ_r$ has canonical singularities then Proposition \ref{proCC} implies that $\ov\phi(x)$ belongs to the unique facet of $\Pi_{\sigma_2,N_2'}$ not containing the origin. This is equivalent  to the fact that  $\sum_{\rho\in \sigma_1(1)}\alpha_{\rho}=1$, which then implies that $x$ also belongs to the unique facet of $\Pi_{\sigma_1,N_1'}$ not containing the origin, i.e.~that
$U_{\sigma_1,N_1'}=U_{\sigma_1,N_1}/\bbZ_r$ has canonical singularities.
\end{proof}

Although we will not use this, just for the sake of completeness, we prove the following criterion for a cyclic quotient of an affine Gorenstein toric variety to be  Gorenstein.

\begin{pro}\label{P:Gorcyc}
Same notation as in Proposition \ref{P:crit-cyc}.
Assume furthermore that $U_{\sigma, N}$ is Gorenstein, so that there is an  $m_\sigma \in  M$  such that
$$ \langle m_\sigma ,u_\rho \rangle =1 \text{ for all } \rho \in \sigma (1),$$
where $u_\rho$ is the primitive element along the ray $\rho$ with respect to the lattice $N$.   If $\lambda$ and $m_\sigma$ satisfy
$$ \frac{1}{r}\lambda(m_\sigma)\in \mathbb Z, $$
then $U_{\sigma,N}/\mathbb Z_r$ is Gorenstein.
\end{pro}
\begin{proof}
We will use the notation of the proof of the above Proposition \ref{P:crit-cyc}. The assumption $ \frac{1}{r}\lambda(m_\sigma)\in \mathbb Z$ implies that $m_\sigma \in M'=(N')^{\vee}$.
Moreover, the fact that $\langle m_\sigma, u_\rho\rangle =1$ insures that $u_\rho$ is  still a primitive generator of $\rho\in \sigma(1)$ with respect to $N'$: indeed if $u_{\rho}=l\cdot \tilde u_{\rho}$ for some $2\leq l\in \bbN$  and
$\tilde u_\rho\in N'$, then
$$1=\langle m_\sigma, u_\rho\rangle =l\langle m_\sigma, \tilde u_\rho\rangle \Rightarrow \langle m_\sigma, \tilde u_\rho\rangle\not\in \bbZ,  $$
which contradicts the fact that $m_\sigma\in (N')^{\vee}$.
\end{proof}

\begin{rem}
If we apply the above Propositions \ref{P:crit-cyc} and \ref{P:Gorcyc} to the case where $U_{\sigma, N}=\bbA^n_k$, we get back one direction of Theorem \ref{T:critsmooth} for finite cyclic quotients of smooth varieties.
\end{rem}

We warn the reader that, contrary to the fact that finite quotients of  $\bbQ$-factorial toric singularities are $\bbQ$-factorial (because the factoriality of $U_{\sigma, N}$ is equivalent to the fact that the cone $\sigma$ is simplicial), a finite quotient  of a Gorenstein toric singularity need not to be $\bbQ$-Gorenstein, as the following example shows.

\begin{exa}\label{E:nonQGor}
 Let $N=\mathbb Z^3=\mathbb Z\langle e_1,e_2,e_3 \rangle $ and consider the toric variety  $U_{\sigma, N}$ defined by the cone
$$
\sigma =\mathbb R_{\ge 0}\langle e_1,e_2,e_3,e_1+e_2-e_3  \rangle \subseteq \mathbb R^3=N\otimes \mathbb R.
$$
Now let $\mathbb Z_2$ act by $-1$ on $x_1$ and as $1$ on $x_2$ and $x_3$.   One can check easily using Propositions \ref{proGC} and \ref{proQGC} that $U_{\sigma,N}$ is Gorenstein, while $U_{\sigma,N}/\mathbb Z_2$ is not $\mathbb Q$-Gorenstein.
\end{exa}

\subsection{Reduction to the cyclic case}\label{SS:redcyc}

In this subsection, we show that in order to detect if a finite quotient $V/G$ of a normal $k$-variety has canonical or terminal singularities, it is enough to  check only that  the cyclic quotients $V/C$ are canonical or terminal  as $C$ varies among all the cyclic subgroups of $G$.  The result in the case where $V$ is smooth appears in a number of places (e.g.~\cite[p.44]{HM}, \cite[Thm.~3.21]{Kol}).   The argument for singular $V$ is the same, and while we expect the result  is well-known in this  case as well, we are unaware of a reference, and so we include the proof here for the convenience of the reader.  

\begin{teo}\label{T:redcyc}
Suppose that $G$ is a finite group acting on $V$, a normal scheme of finite type over  $k$.
Then $V/G$ has canonical (resp.~terminal) singularities if and only if  for every cyclic subgroup $C\le G$, the quotient $V/C$ has canonical (resp.~terminal) singularities.
\end{teo}

\begin{proof}
We will follow the proof of \cite[Thm.~3.21]{Kol}, which deals with the case $V=\bbA_k^n$.
Suppose first that $X=V/G$ does not have canonical (resp.~terminal)  singularities.  Let $\widetilde X\to X$ be a resolution of singularities, and let $E\subseteq \widetilde X$ be a prime divisor such that the discrepancy $a(E,X)<0$ (resp.~$\le 0$).    Let $p:\widetilde V\to \widetilde X$ be the normalization of $\widetilde X$ in the field of fractions of $V$, and let $F\subseteq \widetilde V$ be a prime divisor dominating $E$.   We have a commutative diagram
\begin{equation}\label{E:diag1}
\xymatrix{
\widetilde V \ar@{->}[r]  \ar@{->}[d]_p& V  \ar@{->}[d]\\
\widetilde X \ar@{->}[r]& X=V/G\\
}
\end{equation}
where the vertical morphisms are finite and the horizontal ones are birational.
It is computed in \cite[(2.42.4)]{Kol} that  the discrepancies of $F$ and $E$ are related by the formula
\begin{equation}\label{E:discp}
a(E,X)+1=\frac{a(F,V)+1}{|C_F|}.
\end{equation}
  The group $G$ acts on the field of fractions of $V$, and one can easily check the action preserves integrality, so $G$ also acts on $\widetilde V$ and $\widetilde X=\widetilde V/G$. Let  $C_F$ be the subgroup of $G$ acting as the identity on $F$.  Since $\widetilde V$ is generically smooth along $F$, the subgroup $C_F\le G$ is cyclic.
  The diagram \eqref{E:diag1} factors as follows
  \begin{equation}\label{E:diag2}
\xymatrix{
\widetilde V \ar@{->}[r]  \ar@{->}[d]_q& V   \ar@{->}[d]\\
\widetilde V/C_F \ar@{->}[r] \ar@{->}[d]_r&V/C_F \ar@{->}[d]\\
 \widetilde X =\widetilde V/G \ar@{->}[r]& X=V/G\\
}
\end{equation}
where again the vertical morphisms are finite and the horizontal ones are birational.
Consider the prime divisor $E'=q(F)$, which   is exceptional over  $V/C_F$. By applying formula \eqref{E:discp} to the morphism $q$, we get that
\begin{equation}\label {E:discp2}
a(E',V/C_F)+1=\frac{a(F,V)+1}{|C_F|},
\end{equation}
which together with \eqref{E:discp} implies that
$
a(E',V/C_F)=a(E,X) <0 \ \ (\text{resp.}\ \le 0).
$
Consequently, we see that $V/C_F$ does not have canonical (resp.~terminal) singularities.  

Conversely, suppose there is a cyclic group $C\le G$ such that $V/C$ does not have canonical (resp.~terminal)  singularities.    Let $(V/C)^\sim\to V/C$ be a resolution of singularities, and suppose that $E'$ is an exceptional divisor such that $a(E',V/C)<0$ (resp.~$\le 0$).    Let $\widetilde V$ be  the integral closure of $(V/C)^\sim$ in the field of fractions of $V$, and let $F\subseteq \widetilde V$ be a prime divisor dominating $E'$.    Again we obtain \eqref{E:discp2}.    Now using a result of Zariski and Abhyankar \cite[Lem.~2.22, p.50]{Kol} there is a diagram

\begin{equation}\label{E:diag3}
\xymatrix{
\widetilde V \ar@{->}[r]  \ar@{-->}[d]_p& V  \ar@{->}[d]\\
\widetilde X \ar@{->}[r]& X=V/G\\
}
\end{equation}
where the bottom morphism is birational,  $p$ is the induced  rational map, and  $F$ dominates a prime divisor $E$ of $\widetilde X$.  The computation of  \cite[(2.42.4)]{Kol} holds   (see especially the discussion at the end of the proof of \cite[Cor.~2.43, p.66)]{Kol}), giving \eqref{E:discp}.   Thus we have $a(E,X)=a(E',V/C)<0$ (resp.~$\le 0$), and it follows that $X$ does not have canonical (resp.~terminal) singularities. 
\end{proof}

\bibliography{SingUnivJac}
\end{document}